\RequirePackage{rotating}
\documentclass[12pt]{amsart}
\usepackage[graphicx]{realboxes}
\usepackage{float}
\restylefloat{table}
\usepackage{placeins}

\usepackage[T1]{fontenc}
\usepackage[utf8]{inputenc}
\usepackage{enumitem}
\usepackage{amsmath}
\usepackage{amssymb}
\usepackage{epsfig}
\usepackage{wasysym}
\usepackage{graphicx}
\usepackage{tikz}
\usepackage{tikz-cd}
\usepackage{bm}
\usepackage{xcolor}
\usepackage{listings}
\numberwithin{equation}{section}
\usepackage{extpfeil}
\usepackage{array}
\usepackage{adjustbox}
\usepackage{longtable}
\usepackage{makecell}
\usepackage[all]{xy}
\usepackage{longtable}
\usepackage{rotating}
\input xy
\xyoption{all}
\usepackage{multirow}

\usetikzlibrary{positioning}
\usepackage[colorlinks=false,urlbordercolor=white]{hyperref}
  
\tikzset{sgplattice/.style={inner sep=1pt,norm/.style={red!50!blue},char/.style={blue!50!black},
  lin/.style={black!50}},cnj/.style={black!50,yshift=-2.5pt,left=-1pt of #1,scale=0.5,fill=white}}

\DeclareFontFamily{U}{mathb}{\hyphenchar\font45}
\DeclareFontShape{U}{mathb}{m}{n}{
      <5> <6> <7> <8> <9> <10> gen * mathb
      <10.95> mathb10 <12> <14.4> <17.28> <20.74> <24.88> mathb12
      }{}
\DeclareSymbolFont{mathb}{U}{mathb}{m}{n}
\DeclareMathSymbol{\righttoleftarrow}{3}{mathb}{"FD}

\calclayout
\allowdisplaybreaks[3]

\theoremstyle{plain}
\newtheorem{prop}{Proposition}[section]

\newtheorem{theo}[prop]{Theorem}
\newtheorem{coro}[prop]{Corollary}

\newtheorem{lemm}[prop]{Lemma}

\theoremstyle{definition}
\newtheorem{defi}[prop]{Definition}

\newtheorem*{prob*}{Problem}

\newtheorem{rema}[prop]{Remark}

\newcommand{\actsfromleft}{\mathrel{\reflectbox{$\righttoleftarrow$}}}
\newcommand{\actsfromright}{\righttoleftarrow}

\def\lra{\longrightarrow}

\def\rk{\mathrm{rk}}

\def\fA{{\mathfrak A}}

\def\fC{{\mathfrak C}}
\def\fD{{\mathfrak D}}

\def\fF{{\mathfrak F}}

\def\fS{{\mathfrak S}}

\def\fS{{\mathfrak S}}

\def\bG{{\mathbb G}}
\def\bP{{\mathbb P}}

\def\bF{{\mathbb F}}

\def\Pic{\mathrm{Pic}}

\def\Aut{\mathrm{Aut}}

\def\SL{\mathsf{SL}}
\def\PSL{\mathsf{PSL}}
\def\GL{\mathsf{GL}}

\def\PGL{\mathsf{PGL}}

\def\Out{\mathrm{Out}}

\def\Burn{\mathrm{Burn}}
\def\Bir{\mathrm{Bir}}

\def\lim{\mathrm{lim}}

\def\Cr{\mathrm{Cr}}

\makeatother
\makeatletter

\begin{document}

\title[Actions on del Pezzo surfaces]{Birational geometry of actions \\on del Pezzo surfaces}

\author[I. Cheltsov]{Ivan Cheltsov}
\address{Department of Mathematics, University of Edinburgh, UK}
\email{I.Cheltsov@ed.ac.uk}

\author[Y. Tschinkel]{Yuri Tschinkel}
\address{
  Courant Institute,
  251 Mercer Street,
  New York, NY 10012, USA
}

\email{tschinkel@cims.nyu.edu}

\address{Simons Foundation\\
160 Fifth Avenue\\
New York, NY 10010\\
USA}

\author[Zh. Zhang]{Zhijia Zhang}

\address{
Courant Institute,
  251 Mercer Street,
  New York, NY 10012, USA
}

\email{zz1753@nyu.edu}

\date{\today}

\begin{abstract}
We complete the classification of regular generically free actions of finite groups on del Pezzo surfaces, up to birational equivalence. As a byproduct, we settle several open problems in equivariant birational geometry, e.g., we classify birationally rigid actions on del Pezzo surfaces.
\end{abstract}

\maketitle

\section{Introduction}
\label{sect:intro}

Let $k$ be an algebraically closed field of characteristic zero and 
$$
\Cr_2=\Cr_2(k)=\langle \PGL_3(k), \iota\rangle, 
$$ 
the plane Cremona group; here 
$$
\iota : (x,y,z)\mapsto (\frac{1}{x},\frac{1}{y},\frac{1}{z})
$$
is the Cremona involution. 
The classification of {\em finite subgroups} of $\Cr_2$, up to conjugation,  
has a long history: from the classification of involutions \cite{beauville}, cyclic groups \cite{defernex}, abelian groups \cite{Blanc}, to 
the fundamental paper by Dolgachev--Iskovskikh \cite{DI} classifying finite subgroups of $\Cr_2$. We summarize the methods that allowed one to classify finite groups that can act on rational surfaces: 
\begin{itemize} 
\item 
An embedding of a finite group $G\hookrightarrow \Cr_2$ is given by a regular 
generically free action on a smooth projective rational surface $S$. By the $G$-equivariant Minimal Model Program, we can reduce to the case when $S$ is either  a del Pezzo surface, with invariant Picard rank $\rk\Pic(S)^G=1$, or a $G$-equivariant conic bundle $\pi:S\to \bP^1$ and $\rk\Pic(S)^G=2$. 
\item 
There is an almost complete classification of finite groups acting regularly and generically freely on del Pezzo surfaces \cite{DI}. The classification of automorphisms of conic bundles is more involved, because there are infinitely many families. 
\end{itemize}

The related problem of classification of {\em actions} (raised in, e.g., \cite{popov}) can be approached via the following steps, see \cite[Section 2]{CTZ-linear}:
\begin{itemize}
\item {\bf (A)} classification of  $G\subseteq \Aut(S)$, up to conjugation in $\Aut(S)$, 
\item {\bf (B)} classification of $G\subseteq \Aut(S)$, up to conjugation in $\Cr_2$,
\item {\bf (AA)} classification of actions, up to conjugation in $\Aut(S)$, 
\item {\bf (BA)} classification of actions, up to conjugation in $\Cr_2$. 
\end{itemize}
Assuming that both {\bf (A)} and {\bf (B)} are solved for $G$, we consider homomorphisms 
$$
\phi_1,\phi_2: G\hookrightarrow \Aut(S),
$$ 
with $\phi_1(G)=\phi_2(G)$, such that the {\em actions} given by $\phi_1$ and $\phi_2$ are different, i.e., not conjugated in $\Aut(S)$. 
If they are are conjugated in $\Cr_2$, then 
\begin{equation} 
\label{eqn:main}
[S\actsfromright \phi_1(G)] = [S\actsfromright \phi_2(G)] \in \Burn_2(G),
\end{equation}
where $\Burn_2(G)$ is the equivariant Burnside group, introduced in \cite{BnG}. 
The converse does not always hold, see \cite[Example 7.2]{TYZ-3}, but it holds for {\em intransitive} actions on $S=\bP^2$, by \cite[Theorem 1.1]{CTZ-linear}. 

Let $\Aut^G(S)$ and $\Bir^G(S)$ be  the normalizers of $G$ in the groups of biregular, respectively, 
birational automorphisms of $S$. We have natural homomorphisms
\begin{equation} 
\label{eqn:birg}
\Aut^G(S)\hookrightarrow \Bir^G(S) \stackrel{\bar{\beta}}{\lra}
\Out(G),
\end{equation}
where $\Out(G)$ is the group of outer automorphisms of $G$, 
see \cite[Section 2]{CTZ-linear} for more details.
If 
$$
\bar{\beta}(\Aut^G(S))=\bar{\beta}(\Bir^G(S))
$$
then Problems {\bf (AA)} and {\bf (BA)} coincide and are purely group-theoretic. Therefore, we mostly focus on birational issues, when the images are different. 
In \cite{CTZ-linear}, we addressed Problem {\bf (BA)} for {\em linear actions}, i.e., for subgroups of $\PGL_3$.  Here, we turn to {\em nonlinear} actions, more precisely, regular actions on del Pezzo surfaces that are not (projectively) linearizable \cite{TYZ-3}. We leave out nonlinear actions on conic bundles.

Let $S$ be a del Pezzo surface of anticanonical degree 
$$
d(S):=(-K_X)^2\leq 9,
$$
by definition, smooth. 
Let $G\subseteq \Aut(S)$  be a finite subgroup such that $\rk\Pic(S)^G=1$. Then 
$$
d(S)\in\{1,2,3,4,5,6,8,9\},
$$ 
If $d(S)=9$, then $S=\mathbb{P}^2$,  if $d(S)=8$, then $S=\mathbb{P}^1\times\mathbb{P}^1$.
If $d(S)\ne 9$, then the $G$-action on $S$ is linearizable if and only if it is projectively linearizable, i.e., there exists a $G$-equivariant birational map $S\dasharrow\mathbb{P}^2$. The linearization problem has been settled in~\cite{PSY}.
We record known results from \cite{Isk,DI,PSY,CTZ-linear}:
\begin{itemize} 
\item If $d(S)=1$, then the $G$-action on $S$ is not linearizable, and
$$
\Bir^G(S)=\Aut^G(S),
$$
\item If $d(S)=2$ or $3$, then the $G$-action on $S$ is not linearizable, and
$$
\bar{\beta}(\Bir^G(S))=\bar{\beta}(\Aut^G(S)).
$$
\item If $d(S)=4$, then the $G$-action on $S$ is not linearizable.
\item If  $d(S)=5$, then $\mathrm{Aut}(S)\simeq\mathfrak{S}_5$, and $G$ is one of the following subgroups:
$$
\fC_5, \quad \mathfrak{D}_5, \quad \mathfrak F_5=\mathsf{AGL}_1(\bF_5), \quad \fA_5, \quad \text{ or }  \quad \fS_5.
$$
The $G$-action on $S$ is linearizable if and only if $G\simeq \fC_5$ or $\mathfrak{D}_5$.
Furthermore, if $G\simeq\fA_5$ or $G\simeq\fS_5$, then  $\Bir^G(S)=\Aut^G(S)$.
If $G\simeq\mathfrak F_5$, then $\Out(G)$ is trivial.
\item If $d(S)=6$, then the $G$-action on $S$ is linearizable if and only if $G\simeq \fC_6$ or $\mathfrak{S}_3$.
\item If $d(S)=8$, then the $G$-action on $S$ is linearizable if and only if $S^G\ne\varnothing$, i.e., $G$ fixes a point in $S$.
\item The case $d(S)=9$ has been treated in \cite{CTZ-linear}. 
\end{itemize}

\ 
In particular, Problem {\bf (BA)} is already solved for $d(S)\in\{1,2,3,5,9\}$. 
In this paper, we solve Problem {\bf (BA)} in the remaining cases: 
\begin{itemize}
\item $S=\mathbb{P}^1\times\mathbb{P}^1$, 
\item $S$ is the del Pezzo surface of degree $6$, 
\item $S$ is a del Pezzo surface of degree $4$.
\end{itemize}
In detail, in Section~\ref{sect:quad}, we complete the classification of finite group actions on $S=\bP^1\times \bP^1$, started in \cite{DI,PSY}, and use it to describe generators of $\Bir^G(S)$. 
In Section~\ref{sect:dp6}, we do the same for the del Pezzo surface of degree $6$.
We consider del Pezzo surfaces of degree $4$ in Section~\ref{sect:dp4}.

\ 

Our analysis implies several results of independent interest. The following completes investigations that started in \cite{Segre,Manin-I,DI,Cheltsov2008,two-local,sako,pinardin-solid,yasinski,PSY}: 

\begin{theo}
\label{thm:rigid} 
Let $S$ be a del Pezzo surface and $G\subseteq \Aut(S)$ a finite subgroup such that $\rk\Pic(S)^G=1$. Then $S$ is $G$-birationally rigid if and only if one of the following holds:
\begin{itemize}
    \item $d(S)\le 3$,
    \item $d(S)=4$, and $G\not\simeq$ $\fC_2^2$, $\fD_4$, $\fC_8$, $\fC_2\times \fC_6$, $\fC_3\rtimes \fC_4$, $\fC_3\rtimes \fD_4$, 
    \item $d(S)=5$, and $G\simeq\fA_5, \fS_5$, 
    \item $d(S)=6$, and $G\not\simeq \fC_3\times \fS_3$, $\fC_2\times \fS_3$, $\fS_3$, $\fC_6$,
    \item $S\simeq \bP^1\times \bP^1$, $S$ has no orbits of length $1$ or $2$, and $G\not\simeq \fD_6,\fF_5$, 
    \item $S\simeq\bP^2$, $S$ has no orbits of length 1, and $G\not\simeq \fA_4, \fS_4$. 
\end{itemize}
\end{theo}

The condition $\rk\Pic(S)^G=1$ means that $S$ is a $G$-Mori fiber space (over a point), and $G$-birational rigidity means that $S$ is the unique $G$-Mori fiber space that is $G$-birational to $S$. In \cite{DI}, Dolgachev and Iskovskikh listed all $G$ that can act on a del Pezzo surface $S$ with $\rk\Pic(S)^G=1$. We present a refined classification in Appendix~\ref{sect:tables}, correcting minor inaccuracies. 

 \ 
 
Theorem~\ref{thm:rigid} yields the main result of \cite{yasinski}:

\begin{coro}
\label{coro:yas}
Let $S$ be a del Pezzo surface and $H\subset G\subseteq \Aut(S)$ finite subgroups such that 
$$
\rk\Pic(S)^H=\rk\Pic(S)^G=1.
$$
If $S$ is $H$-birationally rigid, then $S$ is $G$-birationally rigid. 
\end{coro}

Another two byproducts of our analysis are generalizations of the following classical theorem of Segre and Manin \cite{Segre,Manin-I}.

\begin{theo}
Let $G$  be a finite group acting regularly and generically freely on del Pezzo surfaces $S$ and $S'$, with $d(S)=d(S')\leq 3$, and 
$$
\rk\Pic(S)^G=\rk\Pic(S')^G=1.
$$
Then $S$ and $S'$ are $G$-birational if and only if they are $G$-biregular.
\end{theo}

In \cite{shramov-trepalin}, Shramov and Trepalin proved that the Segre--Manin theorem holds for del Pezzo surfaces of degree $4$, see \cite{Elagin} for a categorical proof. In Section~\ref{sect:dp4}, we give an alternative proof of their result. Note that a similar result holds for the del Pezzo surface of degree $5$. In Section~\ref{sect:dp6}, we prove:

\begin{theo}
\label{thm:dP6}
Let $G$  be a finite group acting regularly and generically freely on del Pezzo surfaces $S$ and $S'$, with $d(S)=d(S')=6$, and 
$$
\rk\Pic(S)^G=\rk\Pic(S')^G=1.
$$
Suppose that $G\not\simeq \fC_3\rtimes \fC_6$. Then $S$ and $S'$ are $G$-birational if and only if they are $G$-biregular. \end{theo}

The exceptional case in Theorem~\ref{thm:dP6} was found in \cite{yasinski}. Note however, that the statement fails in the arithmetic setting \cite{KurzYasinsky}. Moreover, the  Segre--Manin theorem  fails for $\mathbb{P}^1\times\mathbb{P}^1$ and $\mathbb{P}^2$. However, it follows from \cite{sako,CTZ-linear} that the Segre--Manin theorem holds for transitive actions on $\mathbb{P}^2$. In Section~\ref{sect:quad}, we generalize this as follows:

\begin{theo}
\label{thm:dP8}
Let $G$  be a finite group acting regularly and generically freely on del Pezzo surfaces $S$ and $S'$ such that  $S\simeq S'\simeq\mathbb{P}^1\times\mathbb{P}^1$, the group $G$ does not fix points in $S$ and $S^\prime$, and 
$$
\rk\Pic(S)^G=\rk\Pic(S')^G=1.
$$
Then $S$ and $S'$ are $G$-birational if and only if they are $G$-biregular. 
\end{theo}

For linearizable actions on $\mathbb{P}^1\times\mathbb{P}^1$, the assertion of Theorem~\ref{thm:dP8} fails in general. For instance, if $G\simeq\mathfrak{D}_4$ acts on $S=\mathbb{P}^1\times\mathbb{P}^1$ as 
\begin{align*}
([x_1:x_2],[y_1:y_2])&\mapsto ([y_1:y_2],[x_1:x_2]),\\    
([x_1:x_2],[y_1:y_2])&\mapsto ([-x_1:x_2],[y_1:y_2]),    
\end{align*}
and $G$ acts on $S^\prime=\mathbb{P}^1\times\mathbb{P}^1$ as 
\begin{align*}
([x_1:x_2],[y_1:y_2])&\mapsto ([y_1:y_2],[x_1:x_2]),\\    
([x_1:x_2],[y_1:y_2])&\mapsto ([ix_1:x_2],[-iy_1:y_2]),    
\end{align*}
then both $G$-actions are linearizable, so it follows from \cite{CTZ-linear} that $S$ and $S^\prime$ are $G$-birational, but $S$ and $S^\prime$ are not $G$-biregular, because they have different fixed loci stratification. In the arithmetic case, the situation is very similar \cite{Trepalin,CT-unpublished}.

\ 

The following result has the same flavor as Theorems~\ref{thm:dP6} and \ref{thm:dP8}: it establishes that, apart from the exceptions listed, birationality of actions coincides with isomorphism of actions.

\begin{theo}
    \label{thm:out}
Let $S$ be a del Pezzo surface of degree $d(S)\leq 8$, and $G\subseteq \Aut(S)$ a finite subgroup such that $\rk\Pic(S)^G=1$, and the $G$-action on $S$ is not linearizable. Then
\begin{equation*} \label{eqn:dre}
\frac{|\bar{\beta}(\Bir^G(S))|}{|\bar{\beta}(\Aut^G(S))|} \le 2.
\end{equation*}
Moreover, equality holds if and only if
\begin{itemize}
    \item $d(S)=4$, and $G\simeq \fC_8$,
    \item $d(S)=6$, and $G\simeq\fS_3^2$,
    \item $S\simeq \bP^1\times \bP^1$, $G\simeq \fC_2\times\fD_n$ with $n$ odd, and $G$ is conjugated to the subgroup generated by 
    $$
     (x,y)\mapsto (\zeta_nx,\zeta_ny),\quad (x,y)\mapsto (\frac1x,\frac1y),\quad (x,y)\mapsto (y,x).
    $$
\end{itemize}
\end{theo}

\ 

\noindent
{\bf Acknowledgments:} 
The first author was supported by Simons Collaboration grant {\em Moduli of varieties}.
The second author was partially supported by NSF grant 2301983.

\section{Quadric}
\label{sect:quad}

Let  $S=\bP^1\times \bP^1$ and 
$$
G\subset \Aut(S)=(\PGL_2)^2\rtimes \fC_2 
$$
be a finite subgroup such that $\rk\Pic(S)^G=1$. Then  $G$ projects nontrivially to $\fC_2$, so that
$$
G\simeq (H\times_Q H)\rtimes \fC_2, \quad H\subset \PGL_2,
$$
where $Q$ is a group and 
$$
H\times_Q H=\{(h,h')\in H\times H: \gamma(h)=\gamma'(h')\},
$$
for surjective homomorphisms $\gamma,\gamma':H\to Q$.
Linearizability of actions on quadric surfaces has been settled in \cite{PSY}. 
We have:  
\begin{itemize} 
\item $H$ is cyclic $\Rightarrow$ the $G$-action is linearizable.
\item $H$ is dihedral $\Rightarrow$ the $G$-action is not linearizable.
\item $H=\fA_4,\fS_4$ or $\fA_5$ $\Rightarrow$ the $G$-action is birationally rigid. 
 \end{itemize}
We focus on nonlinearizable actions, in particular, we assume that $S^G=\varnothing$ and $H$ is
not cyclic. We aim to determine the image of 
$$
\bar{\beta}: \Bir^G(S) \to \mathrm{Out}(G).
$$ 
We summarize the main steps:
\begin{itemize}
    \item Reduction to the case when $H$ is dihedral.
    \item Extraction of rigid and solid $G$-actions.
    \item Classification of non-solid actions. 
\end{itemize}

\subsection*{Reduction to the dihedral case}

If $S$ does not contain $G$-orbits of length less than $6$, then the classification of $G$-Sarkisov links implies that $S$ is either
\begin{itemize}
    \item  $G$-birationally superrigid, or 
     \item $G$-birationally rigid, with $\Bir^G(S)$  generated by $\Aut^G(S)$, Geiser involutions, and Bertini involutions.
\end{itemize}
If $G$ has a $G$-orbit of length $4$, whose points are in general position, then $\Bir^G(S)$ also contains a $G$-birational involution as in:

\begin{lemm}[{\cite[Proposition 5.1]{yasinski}}]
\label{lemma:P1-P1-4-points}
If there is a $G$-Sarkisov link
$$
\xymatrix{
&\tilde{S}\ar[dl]_{\pi}\ar[dr]^{\pi'}&\\
S&&S}
$$
where $\pi$ and $\pi'$ are blowups of $G$-orbits of length $4$ and $\tilde{S}$ a smooth del Pezzo surface of degree $4$, then 
there exists a birational involution $\tau\in\mathrm{Bir}^G(\tilde{S})$ and the following $G$-commutative diagram: 
$$
\xymatrix{
\tilde{S}\ar[d]_{\pi}\ar[rr]^{\phi}&&\tilde{S}\ar[d]^{\pi}\\
S\ar@{-->}[rr]^{\tau}&&S
}
$$
where $\phi$ is a biregular involution of $\tilde{S}$. Moreover, $\tau$ commutes with $G$.
\end{lemm}

\begin{proof}
We give a slightly modified proof of \cite[Proposition 5.1]{yasinski}. By  \cite[Section 6.1]{DI}, there is a natural embedding $\mathrm{Aut}(\tilde{S})\hookrightarrow W_S$, and the Weyl group $W_S\simeq \fC_2^4\rtimes\mathfrak{S}_5$ acts transitively on the set of $(-1)$-curves in $\tilde{S}$. Note also that $W_S$ contains a unique normal subgroup $H\simeq \fC_2^4$, and $H\subset\mathrm{Aut}(\tilde{S})$. Let $E_1,E_2,E_3,E_4$ be the $\pi$-exceptional curves, and $\Gamma$ the stabilizer in $W_S$ of the curve $E_1+E_2+E_3+E_4$. Then $\Gamma\simeq\mathfrak{S}_4\times \fC_2$, and $G\subset \Gamma$, by construction. Using {\tt Magma}, one checks that $H$ contains an involution $\phi$ that commutes with $\Gamma$, as claimed.
\end{proof}

\begin{defi}
An involution $\tau\in\Bir^G(S)$ as in Lemma~\ref{lemma:P1-P1-4-points} will be called a {\em Yasinsky involution} of $S$. 
\end{defi}

Using the classification of $G$-Sarkisov links again, we observe: if $S$ does not have $G$-orbits of length $2$, $3$ or $5$, then either $S$ is $G$-birationally superrigid, and therefore 
$$
\Bir^G(S)=\Aut^G(S),
$$
or $S$ is $G$-birationally rigid and $\Bir^G(S)$ is generated by $\Aut^G(S)$, Geiser involutions, Bertini involutions, and Yasinsky involutions, which also follows from \cite[Corollary 5.2]{yasinski}. Moreover, by \cite[Remark~2.3]{CTZ-linear} and  Lemma~\ref{lemma:P1-P1-4-points}, all these birational involutions lie in the kernel of the homomorphism $\bar{\beta}\colon \Bir^G(S)\to\mathrm{Out}(G)$. This yields: 

\begin{coro}[{cf. \cite[Proposition 6.12]{PSY}}]
\label{corollary:P1-P1-A4}
If  $H=\fA_4,\fS_4$ or $\fA_5$, then $S$ is $G$-birationally rigid, and
$$
\bar{\beta}(\Bir^G(S))=\bar{\beta}(\Aut^G(S)).
$$
\end{coro}

\subsection*{Geometry in the dihedral case}
From now on, we assume that $H$ is dihedral. Then each factor of $S=\mathbb{P}^1\times\mathbb{P}^1$ has an $H$-orbit of length $2$, which is unique if $H\not\simeq \fC_2^2$. Without loss of generality, we may assume that these orbits consist of the points $P_1=[0:1]$ and $P_2=[1:0]$. Put 
\begin{center}
$L_{11}:=\mathrm{pr}_1^*(P_1)$, \quad $L_{12}:=\mathrm{pr}_1^*(P_2)$, \quad $L_{21}:=\mathrm{pr}_2^*(P_1)$, \quad $L_{22}:=\mathrm{pr}_2^*(P_2)$,  
\end{center}
where $\mathrm{pr}_j$ are projections to the factors. 
Then $G$ preserves the torus 
$$
S^\circ:=S\setminus (L_{11}\cup L_{12}\cup L_{21}\cup L_{22})\simeq\mathbb{G}_m^2,
$$
and $G$ is contained in its normalizer in $\mathrm{Aut}(S)$;  we have an exact sequence
\begin{align}\label{eqn:nu}
    1\to G_T\to G\stackrel{\nu}{\lra} \fD_4, 
\end{align}
where $G_T$ is a subgroup of the group of translations in $\bG_m^2$  and $\fD_4$ acts transitively on the set of curves $L_{11},L_{12},L_{21},L_{22}$. 

Since we assume that $\rk\Pic(S)^G=1$ and $G$ does not fix points in $S$, the group $G$ acts transitively on $L_{11},L_{12},L_{21},L_{22}$, and  $\nu(G)$ is: 
$$
\fC_2^2, \quad \fC_4,  \quad  \text{ or } \quad \fD_4.
$$

\subsection*{Rigid groups}

Using the classification of $G$-Sarkisov links \cite{Isk} or, more explicitly, \cite[Corollary 5.2]{yasinski}, we obtain the following:

\begin{lemm}
\label{lemma:P1-P1-rigid}
If $\nu(G)\simeq \fC_4$ or $\fD_4$ as in \eqref{eqn:nu}, and $|G_T|\not\in\{1,2,3,5\}$, then $S$ is $G$-birationally rigid,  and 
$$
\bar{\beta}(\Bir^G(S))=\bar{\beta}(\Aut^G(S)).
$$
\end{lemm}

\begin{proof}
We claim that $S$ does not contain $G$-orbits of length $2$, $3$ or $5$. Indeed, the boundary $L_{11}\cup L_{12}\cup L_{21}\cup L_{22}$  does not contain a $G$-orbits of length $2$, $3$ or $5$, because $\nu(G)$ permutes $L_{11},L_{12},L_{21},L_{22}$ transitively. On the other hand, $G_T$ acts on $S^\circ\simeq\mathbb{G}_m^2$ by translations, so it does not contain $G$-orbits of length $2$, $3$ or $5$, since  $|G_T|\not\in\{1,2,3,5\}$. 

Since $S$ does not contain $G$-orbits of length $2$, $3$ or $5$, as already mentioned, $S$ is $G$-birationally rigid,  and $\Bir^G(S)$ is generated by $\Aut^G(S)$, Geiser involutions, Bertini involutions, and Yasinsky involutions, so 
$$
\bar{\beta}(\Bir^G(S))=\bar{\beta}(\Aut^G(S)),
$$
by \cite[Remark~2.3]{CTZ-linear} and  Lemma~\ref{lemma:P1-P1-4-points}. 
\end{proof}

Recall that $G$ is not cyclic -- otherwise it would fix a point in $S$.
Thus, if $\nu(G)=\fC_4$ and $|G_T|\in\{1,2,3,5\}$, then, up to conjugation, one of the following holds:
\begin{enumerate}
\item $G\simeq  \mathfrak F_5\simeq \fC_5\rtimes \fC_4$ and $G$ is generated by 
$$
(x,y)\mapsto(\zeta_5x,\zeta_5^2y)\quad\text{and}\quad  (x,y)\mapsto(\frac1y,x).
$$
\item $G\simeq \fC_2\times \fC_4$ is generated by 
$$
(x,y)\mapsto(-x,-y)\quad\text{and}\quad (x,y)\mapsto (\frac1y,x).
$$

\end{enumerate}
In the first case, $S$ is $G$-solid \cite{Wolter}, and the $G$-action on $S$ is unique, since $\mathrm{Out}(\mathfrak{F}_5)$ is trivial. In the second case,  $\mathrm{Out}(\fC_2\times \fC_4)\simeq \fC_2^2$, and 
$$
\mathrm{Aut}^G(S)\simeq (\fC_2\times \fC_4)\rtimes\mathfrak{D}_4,
$$ 
with GapID {\tt (64,138)}, generated by 
\begin{align*}
(x,y)& \mapsto (-x,y),\\
 (x,y)& \mapsto (ix,iy),\\
(x,y)& \mapsto (\frac{1}{x},y),\\
 (x,y)& \mapsto (y,x),
\end{align*}
which implies that the $G$-action on $S$ is also unique. In this case, $G$ is conjugated in $\mathrm{Aut}(S)$ to the subgroup $G^\prime$ generated by 
$$
(x,y)\mapsto (\frac{1}{x},\frac{1}{y})\ \text{and}\ (x,y)\mapsto(y,-x).
$$
Since $\nu(G^\prime)\simeq\fC_2^2$,  $S$ is not $G$-solid. In Lemma~\ref{lemm:action-5}, we describe the generators of $\mathrm{Bir}^{G}(S)$ in a more general setting. 

Keeping in mind Lemma~\ref{lemma:P1-P1-rigid}, we need a classification for small $G_T$. Recall that $G_T\subset\mathbb{G}_m^2$ acts on $S\simeq\mathbb{P}^1\times\mathbb{P}^1$ by scaling the first coordinates on each factor. 

\begin{lemm}
If $\nu(G)=\fD_4$ as in \eqref{eqn:nu}, then 
\begin{itemize}
\item either $G_T=\langle (\zeta_n,1),(1,\zeta_n)\rangle\simeq \fC_n^2$ for some $n$, or
\item $n$ is even and $G_T=\langle (\zeta_n,\zeta_n),(\zeta_n^2,1)\rangle\simeq \fC_n\times \fC_{n/2}$.    
\end{itemize}
\end{lemm}

\begin{proof}
Let $g\in G_T$ be an element of the largest order, say $n$. Then $g=(\zeta_n^a,\zeta_n^b)$, where $\gcd(a,b,n)=1$, and $G_T\subset \langle (\zeta_n,1),(1,\zeta_n)\rangle\subset\mathbb{G}_m^2$.  Using the conjugation 
action of $\nu(G)$ on $G_T$, we see that $(\zeta_n^c,\zeta_n)\in G_T$, for some $c$, and also that $(\zeta_n^c,\zeta_n^{-1})\in G_T$, so that
$$
(\zeta_n^c,\zeta_n)(\zeta_n^c,\zeta_n^{-1})^{-1}=(1,\zeta_n^2)\in G_T,
$$
and $(\zeta_n^2,1)\in G_T$ as well. If $c$ is even, then $(1,\zeta_n)\in G_T$, which gives $G_T=\langle (\zeta_n,1),(1,\zeta_n)\rangle$. If $c$ is odd, then $(\zeta_n,\zeta_n)\in G_T$, which gives 
$$
\langle (\zeta_n,\zeta_n),(\zeta_n^2,1)\rangle\subset G_T.
$$
If $n$ is odd, this gives 
$$
G_T=\langle (\zeta_n,1),(1,\zeta_n)\rangle.
$$
If $n$ is even, then 
$$
\langle (\zeta_n,\zeta_n),(\zeta_n^2,1)\rangle\simeq \fC_n\times \fC_{n/2}
$$ 
is a subgroup of index $2$ in $\langle (\zeta_n,1),(1,\zeta_n)\rangle$.
\end{proof}

Applying Lemma~\ref{lemma:P1-P1-rigid}, we obtain:

\begin{coro}
If $\nu(G)=\fD_4$ and $|G_T|\ne 1,2$, then $S$ is $G$-birationally rigid, and 
$\bar{\beta}(\Bir^G(S))=\bar{\beta}(\Aut^G(S))$.
\end{coro}

\begin{proof}
Similar to the proof of Lemma~\ref{lemma:P1-P1-rigid}.
\end{proof}

Assume that $\nu(G)=\fD_4$ and $|G_T|= 1,2$, and $S^G=\varnothing$. Then, up to conjugation, one of the following holds:
\begin{enumerate}
\item $G\simeq \fD_4$, generated by 
$$
(x,y)\mapsto(-y,-x)\quad\text{and}\quad  (x,y)\mapsto(\frac1y,x).
$$
\item $G\simeq \fC_2\times\fD_4\simeq \fC_2^2\rtimes \fC_2^2$, generated by
$$
(x,y)\mapsto(-x,-y),\quad(x,y)\mapsto(y,x),\quad\text{and}\quad  (x,y)\mapsto(\frac1y,x).
$$
\item $G\simeq \fC_2.\fD_4\simeq \fC_2^2\rtimes\fC_4$, with  GapID {\tt (16,3)},   generated by
$$
(x,y)\mapsto(-x,-y),\quad(x,y)\mapsto(y,-x),\quad\text{and}\quad  (x,y)\mapsto(\frac1y,x).
$$
\end{enumerate}
In Case (1), $\mathrm{Out}(\fD_4)\simeq \fC_2$, and  $\mathrm {Aut}^G(S)\simeq (\fC_2\times \fC_4)\rtimes\fC_2^2$, with GapID {\tt (32,49)}, that is generated by $G$, and $(x,y)\mapsto (-x,y)$, and $(x,y)\mapsto (ix,iy)$, which implies that the $G$-action on $S$ is unique. In this case, $G$ is conjugated in $\mathrm{Aut}(S)$ to the subgroup $G^\prime$ generated by 
$$
(x,y)\mapsto (\frac{1}{x},\frac{1}{y})\ \text{and}\ (x,y)\mapsto(iy,ix).
$$
Since $\nu(G^\prime)\simeq\fC_2^2$,  $S$ is not $G$-solid (we will describe the generators of  $\mathrm{Bir}^{G}(S)$ in Lemma~\ref{lemm:action-4}). 

In Case (2), $\mathrm{Out}(\fC_2\times\fD_4)\simeq \fC_2\times\fD_4$, and  $\mathrm {Aut}^G(S)\simeq (\fC_2\times\fC_4)\times\fD_4$, with GapID {\tt (64,138)}. In this case, $G$ is conjugated in $\mathrm{Aut}(S)$ to the subgroup $G^\prime$ generated by 
$$
(x,y)\mapsto (\frac{1}{x},\frac{1}{y}),\quad(x,y)\mapsto(y,x)\ \text{and}\ (x,y)\mapsto(-x,y).
$$
Since $\nu(G^\prime)\simeq\fC_2^2$,  $S$ is not $G$-solid. In Lemma~\ref{lemm:action-1-2}, we will show that $\mathrm{Bir}^{G}(S)=\mathrm{Aut}^{G}(S)$.

In Case (3), $S$ contains a $G$-orbit of length 4, of points 
$$
(i,i),\quad (i,-i),\quad (-i,i),\quad (-i,-i),
$$
and $G$ acts on these points as the cyclic group $\fC_4$, so  
$G$ is conjugated in $\mathrm{Aut}(S)$ to a subgroup $G^\prime$ with $\nu(G')\simeq\fC_4$. Thus $S$ is $G$-birationally rigid, and $\bar{\beta}(\Bir^G(S))=\bar{\beta}(\Aut^G(S))$, by Lemma~\ref{lemma:P1-P1-rigid}. Arguing as in the proof of Lemma~\ref{lemma:P1-P1-rigid}, one can show that $S$ is $G$-birationally superrigid.

\subsection*{Non-solid groups}
If $\nu(G)\simeq \fC_2^2$, then $S$ is not $G$-solid, because $S$ has at least two $G$-orbits of length $2$. One of them is formed by points 
$$
L_{11}\cap L_{21}=(0,0)\quad \text{and}\quad  L_{12}\cap L_{22}=(\infty,\infty),
$$
and another by 
$$
L_{12}\cap L_{21}=(\infty,0)\quad \text{and}\quad L_{11}\cap L_{22}=(0,\infty).
$$
Blowing up one of these orbits $\rho\colon\tilde{S}\to S$,  we obtain a $G$-Sarkisov link: 
$$
\xymatrix{
&\tilde{S}\ar[dl]_{\rho}\ar[dr]^{\pi}&\\
S&&\mathbb{P}^1}
$$
where $\pi$ is a $G$-conic bundle. In particular, $S$ is not $G$-solid, and therefore not $G$-birationally rigid. 


\begin{prop}\label{prop:P1P1gp}
Suppose that $\nu(G)=\fC_2^2$ as in \eqref{eqn:nu}, and $S^G=\varnothing$. 
Then, up to conjugation in $\Aut(S)$, one of the following holds:
\begin{enumerate} 
\item $G$ is generated by 
    $$
    \mathrm{diag}(\zeta_n,1),\quad  \mathrm{diag}(1,\zeta_n),\quad (\frac1x,\frac1y),\quad
(y,x),
    $$
    for $n\geq 2$.
    \item  $G$ is generated by 
    $$
    \mathrm{diag}(\zeta_n,1),\quad  \mathrm{diag}(1,\zeta_n),\quad (\frac1x,\frac1y),\quad
(\zeta_{2n}y,\zeta_{2n}x),
    $$
      for some even $n$.
 \item  $G$ is generated by 
    $$
    \mathrm{diag}(\zeta_n^r,1),\quad  \mathrm{diag}(1,\zeta^r_n),\quad \mathrm{diag}(\zeta_n,\zeta_n),\quad (\frac1x,\frac1y),\quad
(y,x),
    $$
    for some $r,n\geq 2$, where $r\mid n$.
    \item  and $G$ is generated by 
    $$
    \mathrm{diag}(\zeta_n^r,1),\quad  \mathrm{diag}(1,\zeta^r_n),\quad \mathrm{diag}(\zeta_n,\zeta_n),\quad (\frac1x,\frac1y),\quad
(\zeta_{2n}y,\zeta_{2n}x),
    $$
    for some even $n\geq 2$ and  $r\mid n$,  $r\geq 2$.
   \item   $G$ is generated by 
    $$
    \mathrm{diag}(\zeta_n^r,1),\quad  \mathrm{diag}(1,\zeta^r_n),\quad \mathrm{diag}(\zeta_n,\zeta_n),\quad (\frac1x,\frac1y),\quad
(y,\zeta_{2n}^rx),
    $$
    for some even $r,n\geq 2$ where $r\mid n$.
     \item $G$ is generated by 
    $$
    \mathrm{diag}(\zeta_n^r,1),\quad  \mathrm{diag}(1,\zeta^r_n),\quad \mathrm{diag}(\zeta_n,\zeta_n),\quad (\frac1x,\frac1y),\quad
(\zeta_{2n}y,\zeta_{2n}^{1+r}x),
    $$
   for some even $n,r\geq 2$, where  $r\mid n$.
\end{enumerate}
\end{prop}
\begin{proof}
First observe that up, to scaling by the torus, we may assume that $G$ is generated by $G_T$, 
$$
\sigma: (x,y)\to (\frac1x,\frac1y),\quad\text{and}\quad  \tau: (x,y)\to (by,ax),
$$
for some $a,b\in k^\times$. It follows that 
\begin{align}\label{eq:ab}
  \tau^2=\mathrm{diag}(ab,ab)\in G_T,\quad \text{and}\quad (\sigma\tau)^2=\mathrm{diag}(b/a,a/b)\in G_T.  
\end{align}

When $G_T=1$, we know that $ab=1$ and $a/b=1$. It follows that $(a,b)=(1,1)$ or $(-1,-1)$. These two choices are conjugated in $\Aut(S)$ by $(x,y)\mapsto (-x,y)$. In this case, $G$ fixes a point in $S$, contradicting our assumption. 

When $G_T\ne 1$, Goursat's lemma implies the exact sequence 
$$
1\to H\times_Q H\to G\to \fC_2\to 1,
$$
where $H\times_Q H$ acts on $\bP^1\times\bP^1$ without switching the factors. Since $G_T\ne 1$, we know that $H$ is a dihedral group $\fD_n$ of order $2n$, $n\geq 2$. Note that $\fD_2\simeq \fC_2^2$. We know that $Q\ne 1$ since otherwise $\nu(G)=\fD_4$. We have the following possibilities 
$$
Q=\fC_2,\quad \fD_r, \quad \text{where}\quad  r\mid n.
$$
Let $\varphi_1,\varphi_2: \fD_n\to Q$ be the corresponding homomorphisms defining $H\times_Q H$, $K_1=\ker(\varphi_1)$ and $K_2=\ker(\varphi_2)$. Recall that $K_1\times K_2$ is a subgroup of $ H\times_Q H$. Since $\rk\Pic(S)^G=1$, we know that $K_1=K_2$.

When $Q=\fC_2$, if $n$ is even and $K_1=K_2=\fD_{\frac n2}$, then $\bar G=\fD_4$. It follows that $K_1=K_2=\fC_n$, and $G_T=K_1\times K_2$ is generated by 
$$
\mathrm{diag}(\zeta_n,1),\quad\mathrm{diag}(1,\zeta_n).
$$
It follows that $a=\zeta_{2n}^{r_1}$, $b=\zeta_{2n}^{r_2}$, for some integers $r_1,r_2$. Up to multiplying by an element in $G_T$, we may assume that $r_1,r_2\in\{0,1\}$. From \eqref{eq:ab}, we see that $r_1=r_2=0$ or $1$. We obtain Case (2) and (3).

When $Q=\fD_r$, where $r\mid n$, we have $K_1=K_2=\fC_{\frac nr}$, and $G_T$ is generated by 
$$
\mathrm{diag}(\zeta_n^r,1),\quad \mathrm{diag}(1,\zeta_n^r),\quad \mathrm{diag}(\zeta_n,\zeta_n).
$$
As above, we know that $a=\zeta_{2n}^{r_1}$, $b=\zeta_{2n}^{r_2}$ and $r_1+r_2$ is even. Up to multiplying by elements in $G_T$, we may assume that $0\leq r_1\leq r_2<r$, and $r_1\in\{0,1\}$. Recall that $(\sigma\tau)^2=\mathrm{diag}(\zeta_{2n}^{r_2-r_1},\zeta_{2n}^{r_1-r_2})\in G_T$. It follows that $\zeta_{n}^{r_1-r_2}$ is a power of $\zeta_n^r$, and $r_1-r_2=0\pmod r$. Thus, we have at most four possibilities 
$$
(a,b)=(1,1),\quad (\zeta_{2n},\zeta_{2n}), \quad (1,\zeta_{2n}^r),\quad (\zeta_{2n},\zeta_{2n}^{1+r}),
$$
where the latter two are only possible when $r$ is even.
When $n$ is odd, the first two cases give conjugated subgroups in $\Aut(S)$, which correspond to Case (3). When $n$ is even, the four possible cases give non-isomorphic groups in general, which correspond to Cases (3)--(6).
\end{proof}

Using Proposition~\ref{prop:P1P1gp}, we proceed to analyze different actions up to conjugation in the Cremona group. For this, we need to describe generators of $\Bir^G(S)$ in each case. Recall that $S$ contains two $G$-orbits of length $2$: $(0,0)\cup(\infty,\infty)$ and $(0,\infty)\cup(0,\infty)$. Moreover, these are the only $G$-orbits of length $2$ except for one case:

\begin{lemm}
\label{lemm:C23}
Suppose that $G$ is the group described in Case (3) in Proposition~\ref{prop:P1P1gp}, with $n=r=2$. Then $G\simeq \fC_2^3$, $\Out(G)=\GL_3(\bF_2)$, and 
$$
\bar\beta(\Aut^G(S))=\bar\beta(\Bir^G(S))=\fS_4.
$$
In particular, $G$ gives rise to 7 non-birational actions on $S$. 
\end{lemm}

\begin{proof}
We use the Burnside formalism of \cite{BnG}. There are four involutions in $G$ fixing a conic, with a residual $\fC_2^2$-action, giving rise to incompressible symbols 
   $$
   (\fC_2,\fC_2^2\actsfromleft k(\bP^1),(1)).
   $$
   These four involutions sum to 0 in $G$. Up to $\GL_3(\bF_2)$-equivalence, there are 7 choices of such 4 involutions, giving rise to $7$ non-birational actions since their corresponding classes in $\Burn^{\mathrm{inc}}_2(G)$ are different. On the other hand, one can check that $\Aut^G(S)$ contains 
   $$
   (x,y)\mapsto (\frac{x-1}{x+1},\frac{y-1}{y+1}),\quad (x,y)\mapsto (\frac1x,y),\quad(x,y)\mapsto (\zeta_4x,\zeta_4y). 
   $$
   Together with $G$, they generate a group of order 192 with GapID {\tt(192,955)}. The image of this group under $\bar\beta$ in $\Out(G)$ is $\fS_4$. Recall that $|\GL_2(\bF_3)|=168$ and by the Burnside formalism the remaining 7 actions are not birational to each other. It follows that 
   $$
\bar\beta(\Aut^G(S))=\bar\beta(\Bir^G(S))=\fS_4.
   $$
\end{proof}

In fact, arguing as in the proof of Lemma~\ref{lemm:action-3} below, we can also show that $\Aut^G(S)=\Bir^G(S)$ in the case of Lemma~\ref{lemm:C23}.
Similarly, $S$ does not have $G$-orbits of length $3$, except for one classical case \cite{Iskovskikh2003,Iskovskikh2008}.

\begin{lemm}
\label{lemm:D6}
If $G$ is the group in Case (3) in Proposition~\ref{prop:P1P1gp}, with $n=r=3$, then $G\simeq \fD_6$, $\bar\beta(\Aut^G(S))$ is trivial, and
$$
\bar\beta(\Bir^G(S))=\Out(G)\simeq\fC_2.
$$
\end{lemm}

\begin{proof}
Recall that $G\simeq\fD_6\simeq \fC_2\times \fS_3$. Let $\{a,b,c\}$ be a set of generators of $G$, where $a$ is in the center of $G$, $b$ is an order 2 element in $\fS_3$, and  $c$ is an order 3 element in $\fS_3$. Let $\varphi_1:G\to\Aut(S)$ be the homomorphism given by 
$$
(\varphi_1(a))(x,y)= (-x,-y),\quad (\varphi_1(b))(x,y)= (\frac1x,\frac1y), 
$$
$$
(\varphi_1(c))(x,y)= (y,x),
$$
and $\varphi_2:G\to\Aut(S)$ the map given by 
$$
\varphi_2(a)=\varphi_1(a), \quad \varphi_2(b)=\varphi_1(ab),\quad \varphi_1(c)=\varphi_1(c).
$$
Then $\Out(G)$ swaps the actions given by $\varphi_1$ and $\varphi_2$. The map~\eqref{map1} in $\Bir^G(S)$, with $n=2$, also swaps  $\varphi_1$ and $\varphi_2$. Thus 
$$
\bar\beta(\Bir^G(S))=\Out(G).
$$
Using Proposition~\ref{prop:P1P1gp}, we obtain that $\Aut^G(S)$ is generated by 
$G$ and $\mathrm{diag}(-1,-1)$, which implies that $\bar\beta(\Aut^G(S))$ is trivial.
\end{proof}

In the following, we assume that $G$ is not one of the groups described in Lemmas~\ref{lemm:C23} and \ref{lemm:D6}. In particular, $(0,0)\cup(\infty,\infty)$ and $(0,\infty)\cup(\infty,0)$ are the only $G$-orbits of length $2$.  Fix 
$$
\chi\in\Bir^G(S).
$$
If $\chi\not\in\Aut^G(S)$, it can be decomposed into a sequence of $G$-Sarkisov links. On the other hand, every $G$-Sarkisov link that starts at $S$ is given by blowing up $\rho\colon\tilde{S}\to S$ a $G$-orbit $\Sigma\subset S$ such that 
$$
|\Sigma|\in\{2,4,5,6,7\},
$$ 
and $\tilde{S}$ is a del Pezzo surface. 

In fact, $|\Sigma|\ne 5$. Indeed, if $|\Sigma|=5$, it follows that $G\simeq \fC_2\times \fD_5$ is the group described in Case (3) in Proposition~\ref{prop:P1P1gp}, with $n=r=5$. In this case, the only $G$-orbit of length $5$ in $S$ is contained in a curve of degree $(1,1)$, so blowing it up we do not obtain a del Pezzo surface.

If $|\Sigma|=4,6,7$, the $G$-Sarkisov link results in a Yasinsky, Geiser, Bertini involution of the surface $S$, respectively.  Composing $\chi$ with these involutions (if any), we may assume that $|\Sigma|\not\in\{4,6,7\}$.
Thus, $|\Sigma|=2$, and 
$$
\Sigma=(0,0)\cup(\infty,\infty)\quad \text{or}\quad \Sigma=(0,\infty)\cup(\infty,0).
$$
Then $\tilde{S}$ is a smooth del Pezzo surface of degree $6$ with $\mathrm{Pic}(\tilde{S})^G\simeq\mathbb{Z}^2$, and the corresponding $G$-Sarkisov link is:
\begin{equation}
\label{equation:first-link}
\xymatrix{
&\tilde{S}\ar[dl]_{\rho}\ar[dr]^{\pi}&\\
S\ar@{-->}[rr]_{\phi}&&\mathbb{P}^1}
\end{equation}
where  $\pi$ is a conic bundle, and $\phi$ the rational map given by
$$
(x,y)\mapsto\left\{\aligned
&\frac{x}{y}\ \text{if $\Sigma=(0,0)\cup(\infty,\infty)$},\\
&xy\ \text{if $\Sigma=(0,\infty)\cup(\infty,0)$}.
\endaligned
\right.
$$
The second $G$-Sarkisov link used in the decomposition of $\chi$ starts at $\tilde{S}$, and is determined by a $G$-orbit $\tilde{\Sigma}\subset\tilde{S}$ that satisfies the following conditions:
\begin{itemize}
\item[($\diamondsuit$)] each smooth fiber of $\pi$ contains at most one point in $\tilde{\Sigma}$,
\item[($\heartsuit$)] no points of $\tilde{\Sigma}$ are contained in singular fibers of $\pi$,
\end{itemize}
Let $\eta\colon\hat{S}\to\tilde{S}$ be the blow up of $\tilde{\Sigma}$. We expand \eqref{equation:first-link} to the $G$-commutative diagram:
$$
\xymatrix{
&&\hat{S}\ar[dl]_{\eta}\ar[dr]^{\eta^{\prime}}&\\
&\tilde{S}\ar[dl]_{\rho}\ar[dr]^{\pi}&&\tilde{S}^\prime\ar[dl]_{\pi^\prime}\\
S\ar@{-->}[rr]_{\phi}&&\mathbb{P}^1&}
$$
where $\eta^\prime$ is the contraction of the strict transforms of the fibers of $\pi$ that contain points of $\tilde{\Sigma}$, and $\pi^\prime$ is a conic bundle. Moreover, it follows from \cite[Theorem~5]{Iskovskikh1980} that $\tilde{S}^\prime$ is a del Pezzo surface of degree $6$ with $\mathrm{Pic}(\tilde{S}^\prime)^G\simeq\mathbb{Z}^2$. Thus, the third $G$-Sarkisov link used in the decomposition of $\chi$ is 
$$
\xymatrix{
&\tilde{S}^\prime\ar[dr]^{\rho^\prime}\ar[dl]_{\pi^\prime}&\\
\mathbb{P}^1&&S^\prime}
$$
where $S^\prime\simeq \mathbb{P}^1\times\mathbb{P}^1$ with $\mathrm{Pic}(S^\prime)^G=\mathbb{Z}$,
and $\rho^\prime$ is a blow up of a $G$-orbit of length $2$. 
Let $\psi\colon S\dasharrow S^\prime$ be the constructed $G$-birational map. 
Then $\psi$ is a composition of three $G$-Sarkisov links, which are combined in the following $G$-equivariant commutative diagram:
\begin{equation}
\label{quadric}    
\xymatrix{
&&\hat{S}\ar[dl]_{\eta}\ar[dr]^{\eta^{\prime}}&&\\
&\tilde{S}\ar[dl]_{\rho}\ar[dr]^{\pi}&&\tilde{S}^\prime\ar[dl]_{\pi^\prime}\ar[dr]^{\rho^\prime}&\\
S\ar@{-->}[rr]^{\phi}\ar@{-->}@/_1.5pc/[rrrr]_{\psi}&&\mathbb{P}^1&&S^\prime}
\end{equation}
By construction,
$$
\chi=\chi^\prime\circ\psi,
$$
where $\chi^\prime\colon S^\prime\dasharrow S$ is a $G$-birational map that can be decomposed into a fewer number of $G$-Sarkisov links than the original $\chi$.

If the surfaces $S$ and $S^\prime$ in \eqref{quadric} are $G$-biregular, we may choose 
$$
\psi\in\mathrm{Bir}^G(S).
$$
A priori, this may not be the case, cf. \cite[Theorem 2.9]{Trepalin} for the arithmetic counterpart. However, the classification in Proposition~\ref{prop:P1P1gp} implies that this is actually the case:

\begin{prop}
\label{prop:quadrics}
The surfaces $S$ and $S^\prime$ in \eqref{quadric} are $G$-biregular.
\end{prop}

The proof of this result follows from a case by case analysis: we find explicit formulas for $\psi$ in each case and verify that $\psi\in\mathrm{Bir}^G(S)$. This is done in Lemmas~\ref{lemm:action-1-2}, \ref{lemm:action-3}, \ref{lemm:action-4}, \ref{lemm:action-5}, \ref{lemm:action-6} below. 

Note that the construction of $\psi$ may fail for some groups $G$ listed in Proposition~\ref{prop:P1P1gp}. Namely, $\psi$ in \eqref{quadric} is determined by the choice of 
$$
\Sigma=(0,0)\cup(\infty,\infty)\quad\ \text{or}\quad\ \Sigma=(0,\infty)\cup(\infty,0),
$$
and the choice of a $G$-orbit $\tilde{\Sigma}\subset\tilde{S}$ satisfying both ($\diamondsuit$) and ($\heartsuit$). It can happen that $\tilde{S}$ simply does not contain $G$-orbits that satisfy these conditions. 

\begin{rema}
\label{remark:link-exists}
It is easy to check whether or not $\tilde{S}$ contains a $G$-orbit  that satisfies ($\diamondsuit$) and ($\heartsuit$). Since $\phi$ in \eqref{quadric} is \mbox{$G$-equivariant}, it induces an exact sequence of groups:
$$
1\to G_\phi\to G\to G_{\mathbb{P}^1}\to 1,
$$
where $G_{\phi}$ is the kernel of the $G$-action on $\mathbb{P}^1$ and $G_{\mathbb{P}^1}$ is the image of $G$ in $\mathrm{Aut}(\mathbb{P}^1)$. Conditions ($\diamondsuit$) and ($\heartsuit$) imply that $G_{\phi}$ is cyclic. Conversely, if $G_{\phi}$ is cyclic, the $G$-orbit of a general point in $\tilde{S}$ satisfies ($\diamondsuit$) and ($\heartsuit$). 
\end{rema}

Using this observation, we obtain:

\begin{lemm}
\label{lemm:action-1-2}
If $G$ is the group in Cases (1) or (2) in Proposition~\ref{prop:P1P1gp}, then $\Bir^G(S)$ is generated by Bertini, Geiser and Yasinsky involutions, and
$$
\bar{\beta}(\Bir^G(S))=\bar{\beta}(\Aut^G(S)).
$$
Moreover, if $|G|>28$ then $\Bir^G(S)=\Aut^G(S)$.
\end{lemm}

\begin{proof}
As explained, every $\chi\in\Bir^G(S)$ can be decomposed into a composition of $G$-Sarkisov links. Every $G$-Sarkisov link that starts at $S$ gives either a Bertini involution, or a Geiser involution, or a Yasinsky involution, or a $G$-birational map described in \eqref{quadric}. On the other hand, by Remark~\ref{remark:link-exists}, the map in \eqref{quadric} does not exist in our case, because $G_{\phi}$ in Remark~\ref{remark:link-exists} is not cyclic. Indeed, if $\phi$ is 
$$
(x,y)\mapsto \frac{x}{y},
$$
then, in Case (1), $G_\phi$ contains 
$$
(x,y)\mapsto(\zeta_nx,\zeta_ny),\quad (x,y)\mapsto (\frac1y,\frac1x),
$$ 
 and in Case (2), $G_\phi$ contains 
$$
(x,y)\mapsto(\zeta_nx,\zeta_ny),\quad (x,y)\mapsto (\frac{\zeta_{2n}}{y},\frac{\zeta_{2n}}{x}).
$$ 
Similarly, if $\phi$ is 
$$
(x,y)\mapsto xy,
$$
then, in Case (1), $G_\phi$ contains 
$$
(x,y)\mapsto(\zeta_nx,\zeta_n^{-1}y),\quad (x,y)\mapsto (y,x),
$$ 
 and in Case (2), $G_\phi$ contains 
$$
(x,y)\mapsto(\zeta_nx,\zeta_n^{-1}y),\quad (x,y)\mapsto ({\zeta_{2n}^{-1}}y,\zeta_{2n}x).
$$ 
We conclude that $\Bir^G(S)$ is generated by $\Aut^G(S)$, Bertini involutions, Geiser involutions, and Yasinsky involutions. Moreover, by \cite[Remark~2.3]{CTZ-linear} and  Lemma~\ref{lemma:P1-P1-4-points}, these birational involutions lie in the kernel of $\bar{\beta}$, so 
$$
\bar{\beta}(\Bir^G(S))=\bar{\beta}(\Aut^G(S)),
$$
as claimed.

If $|G|>28$, then $S$ does not contain $G$-orbits of length $4$, $6$ and $7$, and the classification of $G$-Sarkisov links implies that $\Bir^G(S)$ does not contain Bertini, Geiser or Yasinsky involutions, so 
$$
\Bir^G(S)=\Aut^G(S).
$$
\end{proof}

We proceed to groups described in Case (3) in Proposition~\ref{prop:P1P1gp}.

\begin{lemm}
\label{lemm:action-3}
Suppose that $G$ is in Case (3) in Proposition~\ref{prop:P1P1gp}. 
\begin{itemize}
\item If $n$ is even or $r\ne n$, then $\Bir^G(S)$ is generated by $\Aut^G(S)$, Bertini involutions, Geiser involutions, and Yasinsky involutions, and 
$$
\bar{\beta}(\Bir^G(S))=\bar{\beta}(\Aut^G(S)).
$$
\item If $n$ is odd and $r=n$, then $\Bir^G(S)$ is generated by  $\Aut^G(S)$, Bertini involutions, Geiser involutions, Yasinsky involutions, and the following birational transformations:
\begin{align}
\label{map1}
        (x_1,x_2)\times(y_1,y_2)\mapsto (r_1,r_2 )\times(t_1,t_2),
     \end{align}
     where 
     \begin{align*}
         r_1&=x_2^{\frac{n+1}{2}}y_1y_2^{\frac{n-1}{2}}-x_1^{\frac{n+1}{2}}y_1^{\frac{n+1}{2}},\qquad
         r_2=x_2^{\frac{n+1}{2}}y_2^{\frac{n+1}{2}}-x_1^{\frac{n-1}{2}}x_2y_1^{\frac{n+1}{2}},\\
         t_1&=x_1x_2^{\frac{n-1}{2}}y_2^{\frac{n+1}{2}}-x_1^{\frac{n+1}{2}}y_1^{\frac{n+1}{2}},\qquad
         t_2=x_2^{\frac{n+1}{2}}y_2^{\frac{n+1}{2}}-x_1^{\frac{n+1}{2}}y_1^{\frac{n-1}{2}}y_2,
     \end{align*}
and 
\begin{align}
\label{map2}
(x_1,x_2)\times(y_1,y_2)\mapsto   
        (f_1,f_2)\times(g_1,g_2),
\end{align}
where
\begin{align*}
    f_1&=-s^nx_1^{n+1}y_1^n + (s^{2n} + 1)x_1^{\frac{n+1}{2}}x_2^{\frac{n+1}{2}}y_1^{\frac{n+1}{2}}y_2^{\frac{n-1}{2}} - s^nx_1x_2^ny_2^n,\\
    f_2&=(s^{2n} + 1)x_1^{\frac{n+1}{2}}x_2^{\frac{n+1}2}y_1^{\frac{n-1}{2}}y_2^{\frac{n+1}{2}} - s^nx_2^{n+1}y_2^n - s^nx_1^nx_2y_1^n,\\
    g_1&=-s^ny_1^{n+1}x_1^n + (s^{2n} + 1)y_1^{\frac{n+1}{2}}y_2^{\frac{n+1}{2}}x_1^{\frac{n+1}{2}}x_2^{\frac{n-1}{2}} - s^ny_1y_2^nx_2^n,\\
    g_2&=(s^{2n} + 1)y_1^{\frac{n+1}{2}}y_2^{\frac{n+1}2}x_1^{\frac{n-1}{2}}x_2^{\frac{n+1}{2}} - s^ny_2^{n+1}x_2^n - s^ny_1^ny_2x_1^n,
\end{align*}
where $s\in k^\times$ and $s^{2n}\ne-1$.
In this case, the image of the map~\eqref{map2} under $\beta$ is trivial, and the image of the map~\eqref{map1} under $\bar{\beta}$ is a nontrivial element in $\Out(G)$ which does not lie in $\bar\beta(\Aut^G(S))$.
\end{itemize}
\end{lemm}

\begin{proof}
The first assertion is proved using the same argument as in the proof of Lemma~\ref{lemm:action-1-2}. We consider the case when $n$ is odd and $r=n$. It suffices  to show that $\psi$ in \eqref{quadric} is a $G$-birational map given by \eqref{map1} or \eqref{map2}. If $\Sigma=(0,0)\cup(\infty,\infty)$, $\tilde{S}$ does not have $G$-orbits that satisfy both ($\diamondsuit$) and ($\heartsuit$), and we see that $\Sigma=(0,\infty)\cup(\infty,0)$,
and $\phi$ in \eqref{quadric} is given by $(x,y)\mapsto xy$. Similarly, 
$$
\tilde{\Sigma}\subset\tilde{C}\subset\tilde{S},
$$ 
where $\tilde{C}$ is the strict transform of $C\subset S$ given by $x=y$. The curve $C$ is pointwise fixed by the involution $(x,y)\mapsto (y,x)$ in $G$. Note that $\tilde{C}$ is a $2$-section of the conic bundle $\pi$. It follows that $\Sigma$ is the $G$-orbit of the strict transform of the point $(s,s)\in\mathbb{P}^1\times\mathbb{P}^1$ with $s\in k^\times$. We have two cases: 
\begin{enumerate}
\item $s^{2n}=1$ and $\Sigma$ has length $n$,
\item $s^{2n}\ne 1$ and $\Sigma$ has length $2n$. 
\end{enumerate}
In the second case, we also have $s^{2n}+1\ne 0$, since otherwise $\tilde{\Sigma}$ does not satisfy ($\diamondsuit$).
In the first case, \eqref{map1} either gives $\psi$ in \eqref{quadric} or $\psi$ conjugated by the involution $(x,y)\mapsto (-x,-y)$, which is contained in $\mathrm{Aut}^G(S)$. In the second case, $\psi$ is given by  \eqref{map2}.

The formulas \eqref{map1} and  \eqref{map2} can be found as follows.  Let $E_1$ and $E_2$ be $\rho$-exceptional curves, let $L_1$ and $L_2$ be general fibers of the projections to the first and the second factor of $S=\mathbb{P}^1\times\mathbb{P}^1$, respectively, let $H_1$ and $H_2$ be the strict transforms on $\tilde{S}$ of the curves $L_1$ and $L_2$, let $E_1^\prime$ and $E_2^\prime$ be the $\rho^\prime$-exceptional curves, let $L_1^\prime$ and $L_2^\prime$ be general fibers of the projections to the first and the second factor of $S^\prime=\mathbb{P}^1\times\mathbb{P}^1$, respectively, let $H_1^\prime$ and $H_2^\prime$ be the strict transforms on $\tilde{S}^\prime$ of the curves $L_1$ and $L_2$, and let $\tilde{E}_1^\prime$, $\tilde{E}_2^\prime$, $\tilde{L}_1^\prime$, $\tilde{L}_2^\prime$ be the strict transforms on $\tilde{S}$ of the curves $E_1^\prime$, $E_2^\prime$, $L_1^\prime$, $L_2^\prime$, respectively. Then 
$$
\tilde{E}_1^\prime\sim aH_1+bH_2-m_1E_1-m_2E_2
$$
for some non-negative integers $a$, $b$, $m_1$, $m_2$. Moreover, since $\tilde{E}_1^\prime$ and $\tilde{E}_2^\prime$ are swapped by $G$-action, we see that
$$
\tilde{E}_2^\prime\sim bH_1+aH_2-m_2E_1-m_1E_2.
$$
Furthermore, by construction, 
$$
a^2+b^2-2m_1m_2=\tilde{E}_1^\prime\cdot \tilde{E}_2^\prime=|\Sigma|
$$ 
and 
$$
2ab-m_1^2-m_2^2=\big(\tilde{E}_1^\prime\big)^2=\big(\tilde{E}_1^\prime\big)^2=-1+|\Sigma|.
$$
Since $\tilde{E}_1^\prime$ and $\tilde{E}_1^\prime$ are sections of $\pi$, we also have
$$
a+b-m_1-m_2=\tilde{E}_1^\prime\cdot(H_1+H_2-E_1-E_2)=\tilde{E}_2^\prime\cdot(H_1+H_2-E_1-E_2)=1.
$$
Recall that the conic bundle $\pi$ has two reducible fibers (over the points $0$ and $\infty$).
Let $\ell_{0}$, $\ell_{0}^\prime$, $\ell_{\infty}$, $\ell_{\infty}^\prime$ be their irreducible components such that $\pi(\ell_{0})=\pi(\ell_{0}^\prime)=0$,  $\pi(\ell_{\infty})=\pi(\ell_{\infty}^\prime)=\infty$. Then we may assume that $\ell_{0}\sim H_1-E_1$, $\ell_{0}^\prime\sim H_2-E_2$, $\ell_{\infty}\sim H_2-E_1$, $\ell_{\infty}^\prime\sim H_1-E_2$, so $\ell_{0}\cap E_1\ne\varnothing$, $\ell_{0}^\prime\cap E_2\ne\varnothing$, $\ell_{\infty}\cap E_1\ne\varnothing$, $\ell_{\infty}^\prime\cap E_2\ne\varnothing$. Without loss of generality, we may assume that 
$\tilde{E}_1^\prime\cap\ell_0\ne\varnothing$ and $\tilde{E}_2^\prime\cap\ell_0^\prime\ne\varnothing$, so 
$$
\tilde{E}_1^\prime\cdot\ell_0=1,\quad  \tilde{E}_1^\prime\cdot\ell_0^\prime=0, \quad \tilde{E}_2^\prime\cdot\ell_0=0, \quad \tilde{E}_2^\prime\cdot\ell_0^\prime=1.
$$
This gives us
$$
b-m_1=\tilde{E}_2^\prime\cdot\ell_0^\prime=\tilde{E}_1^\prime\cdot\ell_0=1
$$
and
$$
a-m_2=\tilde{E}_2^\prime\cdot\ell_0=\tilde{E}_1^\prime\cdot\ell_0^\prime=0.
$$
Hence, we have one of the following possibilities:
\begin{enumerate}
\item 
either $\tilde{E}_1^\prime\cdot\ell_\infty=1$,  $\tilde{E}_1^\prime\cdot\ell_\infty^\prime=0$, $\tilde{E}_2^\prime\cdot\ell_\infty=0$, $\tilde{E}_2^\prime\cdot\ell_\infty^\prime=1$, or
\item 
$\tilde{E}_1^\prime\cdot\ell_\infty=0$, $\tilde{E}_1^\prime\cdot\ell_\infty^\prime=1$, $\tilde{E}_2^\prime\cdot\ell_\infty=1$, $\tilde{E}_2^\prime\cdot\ell_\infty^\prime=0$.
\end{enumerate}
In the first case, 
$$
a-m_1=\tilde{E}_2^\prime\cdot\ell_\infty^\prime=\tilde{E}_1^\prime\cdot\ell_\infty=1
$$
and
$$
b-m_2=\tilde{E}_2^\prime\cdot\ell_\infty=\tilde{E}_1^\prime\cdot\ell_\infty^\prime=0,
$$
which gives $a=b=m_2=\frac{|\tilde{\Sigma}|}{2}$ and $m_1=\frac{|\tilde{\Sigma}|-2}{2}$, so, in particular, $|\tilde{\Sigma}|$ is even. This gives
$$
\tilde{H}_1^\prime\sim \frac{|\tilde{\Sigma}|+2}{2}H_1+\frac{|\tilde{\Sigma}|}{2}H_2-\frac{|\tilde{\Sigma}|}{2}(E_1+E_2),
$$
and 
$$
\tilde{H}_2^\prime\sim \frac{|\tilde{\Sigma}|}{2}H_1+\frac{|\tilde{\Sigma}|+2}{2}H_2-\frac{|\Sigma|}{2}(E_1+E_2),
$$
because 
$$
H_1+H_2-E_1-E_2\sim \tilde{H}_1^\prime+\tilde{H}_2^\prime-\tilde{E}_1^\prime-\tilde{E}_2^\prime.
$$
In the second case, we have
$$
a-m_1=\tilde{E}_2^\prime\cdot\ell_\infty^\prime=\tilde{E}_1^\prime\cdot\ell_\infty=0
$$
and
$$
b-m_2=\tilde{E}_2^\prime\cdot\ell_\infty=\tilde{E}_1^\prime\cdot\ell_\infty^\prime=1,
$$
which gives $b=m_1=m_2=\frac{|\Sigma|-1}{2}$ and $a=\frac{|\Sigma|+1}{2}$. In particular, $|\Sigma|$ is odd.
This gives
$$
\tilde{H}_1^\prime\sim \frac{|\tilde{\Sigma}|+1}{2}(H_1+H_2)-\frac{|\tilde{\Sigma}|+1}{2}E_1-\frac{|\tilde{\Sigma}|-1}{2}E_2,
$$
and 
$$
\tilde{H}_2^\prime\sim\frac{|\tilde{\Sigma}|+1}{2}(H_1+H_2)-\frac{|\tilde{\Sigma}|-1}{2}E_1-\frac{|\tilde{\Sigma}|+1}{2}E_2.
$$
Now, we can find pencils in the linear systems $|\tilde{H}_1^\prime|$ and $|\tilde{H}_2^\prime|$ that consist of curves passing through $\tilde{\Sigma}$. Choosing an appropriate basis in each of these pencils, we obtain an explicit equation for $\psi$ in \eqref{quadric}. In particular, if $s=1$, this map is given by \eqref{map1}. If $s=-1$, this map is given by \eqref{map1}, conjugated by the involution $(x,y)\mapsto (-x,-y)$. Finally, if $s\ne \pm 1$, then $\psi$ is given by \eqref{map2}. 

Note that when $n=r$ is odd, $G\simeq C_2\times\fD_n$. One can check that~\eqref{map2} commutes with the actions, and thus lies in the kernel of $\beta$. Let $\{a,b,c\}$ be a set of generators of $G$, where $a$ is in the center of $G$, $b$ is an order $2$ element in $\fD_n$, and $c$ is an order $n$ element in $\fD_n$.  The image of \eqref{map1} under $\bar\beta$ is the outer automorphism of $G$ given by 
$$
a\mapsto a,\quad b\mapsto ba,\quad c\mapsto c. 
$$
On the other hand, $\Aut^G(S)$ is generated by $G$ and translation elements in the torus, which cannot give rise to the above outer automorphism.
\end{proof}

We proceed to groups described in Case (4) in Proposition~\ref{prop:P1P1gp}.

\begin{lemm}
\label{lemm:action-4}
If $G$ is the group in Case (4) in Proposition~\ref{prop:P1P1gp} then
$$
\bar{\beta}(\Bir^G(S))=\bar{\beta}(\Aut^G(S)).
$$
Moreover, the following assertions hold:
\begin{itemize}
    \item If $r$ is odd, $\Bir^G(S)$ is generated by $\Aut^G(S)$, Bertini involutions, Geiser involutions, and Yasinsky involutions.

    \item If $r$ is even, $\Bir^G(S)$ is generated by $\Aut^G(S)$, Bertini involutions, Geiser involutions, Yasinsky involutions,  and the following birational transformations:
\begin{equation}
\label{map4a}    
(x_1,x_2)\times(y_1,y_2)\mapsto (r_1,r_2)\times(t_1,t_2)
\end{equation}
      where 
      \begin{align*}
          r_1&=x_1(x_1^{\frac n2}y_1^{\frac n2}+\zeta_4x_2^{\frac n2}y_2^{\frac n2}),\qquad
          r_2=x_2(x_2^{\frac n2}y_2^{\frac n2}+\zeta_4x_1^{\frac n2}y_1^{\frac n2}), \\
          t_1&=y_1(x_1^{\frac n2}y_1^{\frac n2}-\zeta_4x_2^{\frac n2}y_2^{\frac n2}),\qquad
          t_2=y_2(x_2^{\frac n2}y_2^{\frac n2}-\zeta_4x_1^{\frac n2}y_1^{\frac n2}),
      \end{align*}
      and 
\begin{equation}
\label{map4b}    
(x_1,x_2)\times(y_1,y_2)\mapsto (f_1,f_2)\times(g_1,g_2)
\end{equation}
      where 
      \begin{align*}
          f_1&=s^{\frac n2}y_2^nx_1x_2^n+(1-s^n)y_1^{\frac n2}y_2^{\frac n2}x_2^{\frac n2}x_1^{\frac{n}{2}+1}-s^{\frac n2}y_1^nx_1^{n+1},\\
      f_2&=s^{\frac n2}y_1^nx_2x_1^n+(1-s^n)y_2^{\frac n2}y_1^{\frac n2}x_1^{\frac n2}x_2^{\frac{n}{2}+1}-s^{\frac n2}y_2^nx_2^{n+1},\\
       g_1&=s^{\frac n2}x_2^ny_1y_2^n-(1-s^n)x_1^{\frac n2}x_2^{\frac n2}y_2^{\frac n2}y_1^{\frac{n}{2}+1}-s^{\frac n2}x_1^ny_1^{n+1},\\
      g_2&=s^{\frac n2}x_1^ny_2y_1^n-(1-s^n)x_2^{\frac n2}x_1^{\frac n2}y_1^{\frac n2}y_2^{\frac{n}{2}+1}-s^{\frac n2}x_2^ny_2^{n+1},
      \end{align*}
where $s\in k^\times$ such that $s^{2n}\ne1$. 
\end{itemize}
\end{lemm}

\begin{proof}
Arguing as in the proof of Lemma~\ref{lemm:action-1-2}, we may assume that $r$ is even. It is enough to show that $\psi$ in \eqref{quadric} is given by \eqref{map4a} or \eqref{map4b}, which can be done arguing as in the proof of Lemma~\ref{lemm:action-3}, so we will use the notation introduced in the proof of that lemma. The only difference is that now $\tilde{\Sigma}$ is contained in the union of $\rho$-exceptional curves $E_1\cup E_2$, so we have $\tilde{E}_1^\prime=E_1$ and $\tilde{E}_2^\prime=E_2$. This gives
$$
\tilde{H}_1^\prime\sim \frac{|\tilde{\Sigma}|+2}{2}H_1+\frac{|\tilde{\Sigma}|}{2}H_2-\frac{|\tilde{\Sigma}|}{2}(E_1+E_2),
$$
and 
$$
\tilde{H}_2^\prime\sim \frac{|\tilde{\Sigma}|}{2}H_1+\frac{|\tilde{\Sigma}|+2}{2}H_2-\frac{|\tilde{\Sigma}|}{2}(E_1+E_2),
$$
Hence, as in the proof of Lemma~\ref{lemm:action-3}, to find explicit equation of the map $\psi$ in \eqref{quadric}, we have to find pencils in $|\tilde{H}_1^\prime|$ and $|\tilde{H}_2^\prime|$ consisting of curves that pass through $\tilde{\Sigma}$. 

The conic bundle $\pi$ gives us isomorphisms $E_1\simeq\mathbb{P}^1$ and $E_2\simeq\mathbb{P}^1$. Let us use these (explicit) isomorphisms to identify $E_1=\mathbb{P}^1$ and $E_2=\mathbb{P}^1$. Fix $s\in E_1\cap \tilde{\Sigma}$, and let $F_s$ be the fiber of $\pi$ that passes through $s$. Then $s\not\in\{0,\infty\}$, and $F_s$ is the strict transform of the curve in $S$ given by $xy=s$. Set $m=\frac{n}{2}$. Then $\tilde{\Sigma}\cap E_1$ consists of the points 
$$
s, \zeta_m s, \zeta_m^2 s, \cdots, \zeta_m^{m-1} s, \frac{\zeta_n}{s}, \zeta_m\frac{\zeta_n}{s}, \zeta_m^2\frac{\zeta_n}{s}, \cdots, \zeta_m^{m-1}\frac{\zeta_n}{s},
$$
and the intersection $\tilde{\Sigma}\cap E_2$ consists of the points 
$$
\zeta_n s, \zeta_m \zeta_n s, \zeta_m^2 \zeta_n s, \cdots, \zeta_m^{m-1} \zeta_n s, \frac{1}{s}, \zeta_m\frac{1}{s}, \zeta_m^2\frac{1}{s}, \cdots, \zeta_m^{m-1}\frac{1}{s}.
$$
Since $\tilde{\Sigma}$ satisfies ($\diamondsuit$), we have $E_2\cap F_s\not\in\tilde{\Sigma}$, which gives $s^{n}\ne 1$. 
Moreover, $|\tilde{\Sigma}|=2n$ unless $s^{2n}=1$. If $s^{2n}=1$, then $|\tilde{\Sigma}|=n$. As in the proof of Lemma~\ref{lemm:action-3}, we find pencils in the linear systems $|\tilde{H}_1^\prime|$ and $|\tilde{H}_2^\prime|$ that consist of curves passing through $\tilde{\Sigma}$, and choose appropriate basis in each of these pencils. This gives us explicit equations for the map $\psi$ in each case.
\end{proof}


\begin{lemm}
\label{lemm:action-5}
If $G$ is the group in Case (5) in Proposition~\ref{prop:P1P1gp} then
$$
\bar{\beta}(\Bir^G(S))=\bar{\beta}(\Aut^G(S)).
$$
Moreover, the following assertions hold:
\begin{itemize}

      \item If $r\equiv 0\pmod 4$ then $\Bir^G(S)$ is generated by $\Aut^G(S)$, Bertini involutions, Geiser involutions, Yasinsky involutions, the birational transformations
      \begin{equation}
\label{map5c}    
(x_1,x_2)\times(y_1,y_2)\mapsto (r_1,r_2)\times(t_1,t_2),
\end{equation}
      where 
      \begin{align*}
          r_1&=x_1(x_1^ky_2^k-\zeta_4y_1^kx_2^k),\qquad
          r_2=x_2(x_2^ky_1^k-\zeta_4y_2^kx_1^k), \\
          t_1&=y_1(y_1^kx_2^k+\zeta_4x_1^ky_2^k),\qquad
          t_2=y_2(-y_2^kx_1^k-\zeta_4x_2^ky_1^k),
      \end{align*} 
with $k=\frac nr$, and  the birational transformations
\begin{equation}
\label{map5d}    
(x_1,x_2)\times(y_1,y_2)\mapsto (f_1,f_2)\times(g_1,g_2)
\end{equation}
      where 
      \begin{align*}
          f_1&=x_1(s^kx_1^{2k}y_2^{2k}+(1-s^{2k})x_1^ky_1^kx_2^ky_2^k-s^ky_1^{2k}x_2^{2k}),\\
          f_2&=x_2(s^kx_2^{2k}y_1^{2k}+(1-s^{2k})x_2^ky_2^kx_1^ky_1^k-s^ky_2^{2k}x_1^{2k}),\\
           g_1&=y_1(s^ky_1^{2k}x_2^{2k}-(1-s^{2k})y_1^kx_1^ky_2^kx_2^k-s^kx_1^{2k}y_2^{2k}),\\
            g_2&=y_2(s^ky_2^{2k}x_1^{2k}-(1-s^{2k})y_2^kx_2^ky_1^kx_1^k-s^kx_2^{2k}y_1^{2k}),
      \end{align*} 
      with $k=\frac nr$, and $s\in k^\times$ such that $s^{4k}\ne1$.

\item If $r\equiv 2\pmod 4$ then $\Bir^G(S)$ is generated by $\Aut^G(S)$, Bertini involutions, Geiser involutions, Yasinsky involutions, the birational transformations~\eqref{map5c} and ~\eqref{map5d}, and 
 the birational transformations
\begin{equation}
\label{map5a}    
(x_1,x_2)\times(y_1,y_2)\mapsto (r_1,r_2)\times(t_1,t_2),
\end{equation}
      where 
      \begin{align*}
          r_1&=x_1(x_1^{\frac n2}y_1^{\frac n2}+\zeta_4x_2^{\frac n2}y_2^{\frac n2}),\qquad
          r_2=x_2(x_2^{\frac n2}y_2^{\frac n2}+\zeta_4x_1^{\frac n2}y_1^{\frac n2}), \\
          t_1&=y_1(x_1^{\frac n2}y_1^{\frac n2}-\zeta_4x_2^{\frac n2}y_2^{\frac n2}),\qquad
          t_2=y_2(x_2^{\frac n2}y_2^{\frac n2}-\zeta_4x_1^{\frac n2}y_1^{\frac n2}),
      \end{align*} 
      and
\begin{equation}
\label{map5b}    
(x_1,x_2)\times(y_1,y_2)\mapsto (f_1,f_2)\times(g_1,g_2),
\end{equation}
      where 
      \begin{align*}
          f_1&=s^{\frac n2}y_2^nx_1x_2^n+(1-s^n)y_1^{\frac n2}y_2^{\frac n2}x_2^{\frac n2}x_1^{\frac{n}{2}+1}-s^{\frac n2}y_1^nx_1^{n+1},\\
      f_2&=s^{\frac n2}y_1^nx_2x_1^n+(1-s^n)y_2^{\frac n2}y_1^{\frac n2}x_1^{\frac n2}x_2^{\frac{n}{2}+1}-s^{\frac n2}y_2^nx_2^{n+1},\\
       g_1&=s^{\frac n2}x_2^ny_1y_2^n-(1-s^n)x_1^{\frac n2}x_2^{\frac n2}y_2^{\frac n2}y_1^{\frac{n}{2}+1}-s^{\frac n2}x_1^ny_1^{n+1},\\
      g_2&=s^{\frac n2}x_1^ny_2y_1^n-(1-s^n)x_2^{\frac n2}x_1^{\frac n2}y_1^{\frac n2}y_2^{\frac{n}{2}+1}-s^{\frac n2}x_2^ny_2^{n+1},
      \end{align*}
for $s\in k^\times$ such that $s^{2n}\ne1$.
\end{itemize}

\end{lemm}

\begin{proof}
We use notation introduced in the proof of Lemma~\ref{lemm:action-4}. As in the proof of Lemma~\ref{lemm:action-4}, it suffices to consider the case when $r$ is even. The only difference is that now we have the following two possibilities:
\begin{enumerate}
\item either $\rho$ is the blow up of the $G$-orbit $(0,\infty)\cup (\infty,0)$ and $\phi$ in \eqref{quadric} is given by $(x,y)\mapsto xy$ (as in the proof of Lemmas~\ref{lemm:action-3} and \ref{lemm:action-4}), 
\item or $\rho$ is the blow up of the $G$-orbit $(0,0)\cup (\infty,\infty)$ and $\phi$ is given by $(x,y)\mapsto \frac{x}{y}$.
\end{enumerate}
In the first case, we proceed exactly as in the proof of  Lemma~\ref{lemm:action-4} to obtain the explicit equation of the birational map $\psi$ in \eqref{quadric}, which gives us one of the birational maps \eqref{map5a} or \eqref{map5b}. Let us deal with the second case.  

As in the proof of  Lemma~\ref{lemm:action-4}, we have $\tilde{\Sigma}\subset E_1\cap E_2$, which gives $\tilde{E}_1^\prime=E_1$ and $\tilde{E}_2^\prime=E_2$. This gives
$$
\tilde{H}_1^\prime\sim \frac{|\tilde{\Sigma}|+2}{2}H_1+\frac{|\tilde{\Sigma}|}{2}H_2-\frac{|\tilde{\Sigma}|}{2}(E_1+E_2),
$$
and 
$$
\tilde{H}_2^\prime\sim \frac{|\tilde{\Sigma}|}{2}H_1+\frac{|\tilde{\Sigma}|+2}{2}H_2-\frac{|\tilde{\Sigma}|}{2}(E_1+E_2),
$$
Similarly, we identify $E_1=\mathbb{P}^1$ and $E_2=\mathbb{P}^1$ using the isomorphisms $E_1\simeq\mathbb{P}^1$ and $E_2\simeq\mathbb{P}^1$ induced by the conic bundle $\pi$. Fix $s\in E_1\cap \tilde{\Sigma}$, and let $F_s$ be the fiber of $\pi$ that passes through $s$. Then $s\not\in\{0,\infty\}$, and $F_s$ is the strict transform of the curve in $S$ given by $x=sy$. Set $m=\frac{n}{r}$. Then $\tilde{\Sigma}\cap E_1$ consists of the points 
$$
s, \zeta_m s, \zeta_m^2 s, \cdots, \zeta_m^{m-1} s, \frac{\zeta_{2m}}{s}, \zeta_m\frac{\zeta_{2m}}{s}, \zeta_m^2\frac{\zeta_{2m}}{s}, \cdots, \zeta_m^{m-1}\frac{\zeta_{2m}}{s},
$$
and the intersection $\tilde{\Sigma}\cap E_2$ consists of the points 
$$
\frac{1}{s}, \zeta_m \frac{1}{s}, \zeta_m^2 \frac{1}{s}, \cdots, \zeta_m^{m-1} \frac{1}{s}, \zeta_{2m}s, \zeta_m\zeta_{2m}s, \zeta_m^2\zeta_{2m}s, \cdots, \zeta_m^{m-1}\zeta_{2m}s.
$$
Hence, $s^{2m}\ne 1$, because $E_2\cap F_s\not\in\tilde{\Sigma}$.
We have that $|\tilde{\Sigma}|=4m$ when $s^{4m}\ne 1$, and $|\tilde{\Sigma}|=2m$ when $s^{4m}=1$. We can explicitly find pencils in the linear systems $|\tilde{H}_1^\prime|$ and $|\tilde{H}_2^\prime|$ that consist of curves passing through $\tilde{\Sigma}$, choose appropriate basis in each of these pencils, and obtain an explicit equation for the birational map $\psi$ in \eqref{quadric} in each case. This gives the maps \eqref{map5c} and \eqref{map5d}.
\end{proof}

Arguing as in the proofs of Lemmas~\ref{lemm:action-3}, \ref{lemm:action-4}, \ref{lemm:action-5}, we obtain

\begin{lemm}
\label{lemm:action-6}
If $G$ is the group in Case (6) in Proposition~\ref{prop:P1P1gp} then
$$
\bar{\beta}(\Bir^G(S))=\bar{\beta}(\Aut^G(S)).
$$
Moreover, the following assertions hold:
\begin{itemize}
      \item If $r\equiv 2\pmod 4$, then $\Bir^G(S)$ is generated by $\Aut^G(S)$, Bertini involutions, Geiser involutions, Yasinsky involutions, and birational transformations~\eqref{map5c} and~\eqref{map5d}.
       \item If $r\equiv 0\pmod 4$, then $\Bir^G(S)$ is generated by $\Aut^G(S)$, Bertini involutions, Geiser involutions, Yasinsky involutions, and birational transformations~\eqref{map5c},~\eqref{map5d},~\eqref{map5a}, and~\eqref{map5b}.
\end{itemize}
\end{lemm}

\section{Degree 6 Del Pezzo surfaces}
\label{sect:dp6}

Let $S$ be the degree 6 del Pezzo surface, i.e., the blowup of $\bP^2$ in 
$$
[1:0:0], \quad [0:1:0], \quad [0:0:1]. 
$$
Its automorphism group fits into a split exact sequence
$$
1\to \bG_m^2\to \Aut(S) \stackrel{\nu}{\lra}  \fS_3\times \fC_2 \to 1.
$$
We identify $\Aut(S)$ with the group generated by the 
natural $\bG_m^2$-action on $\bP^2_{x_1,x_2,x_3}$,  a permutation action of 
$$
\fS_3=\langle \sigma_{123},\sigma_{12}\rangle,
$$
$$
\sigma_{123}: (x_1,x_2,x_3)\mapsto (x_2,x_3,x_1), \quad  \sigma_{12}: (x_1,x_2,x_3)\mapsto (x_2,x_1,x_3),
$$ 
and the standard Cremona involution $\iota$.

\subsection*{Classification of groups}
We proceed with the description of finite subgroups $G\subset \Aut(S)$, 
following \cite[Theorem 6.3]{DI}.  
 We have an exact sequence
$$
1\to G_T\to G\stackrel{\nu}{\lra} \fS_3\times \fC_2\simeq \fD_6.  
$$
Assuming that $\rk\Pic(S)^G=1$, 
we have
$$
\nu(G)=\fS_3\times \fC_2,  \quad \fC_3\times \fC_2, \text{ or } \fS_3',
$$
where $\fS_3'$ is
the {\em twisted}, by the center, subgroup of $\fD_6$. 

\begin{prop}
\label{prop:dp6}
Up to conjugation in $\Aut(S)$, one of the following holds:
\begin{enumerate}
    \item $G\simeq \fC_n^2\rtimes(\fS_3\times \fC_2)$  is generated by 
    $$
\mathrm{diag}(\zeta_n, 1, 1), \quad \mathrm{diag}(1, \zeta_n, 1), \quad \sigma_{123},\quad  \sigma_{12}, \quad \iota.
    $$
   \item $G\simeq (\fC_n\times \fC_{n/3})\rtimes(\fS_3\times \fC_2)$ is generated by 
    $$
\mathrm{diag}(\zeta_n^3, 1, 1), \quad \mathrm{diag}(\zeta_n^2, \zeta_n, 1), \quad \sigma_{123},\quad  \sigma_{12}, \quad \iota,
    $$   
 where $3\mid n$.

      \item $G\simeq \fC_n^2\rtimes\fS_3$ is generated by 
    $$
\mathrm{diag}(\zeta_n, 1, 1), \quad \mathrm{diag}(1, \zeta_n, 1), \quad \sigma_{123}, \quad \sigma_{12}\cdot \iota. 
    $$   
   \item $G\simeq (\fC_n\times \fC_{n/3})\rtimes\fS_3$ is generated by 
    $$
\mathrm{diag}(\zeta_n^3, 1, 1), \quad \mathrm{diag}(\zeta_n^2, \zeta_n, 1), \quad \sigma_{123}, \quad \sigma_{12}\cdot \iota,  
    $$   
    where $3\mid n$.  
      \item $G\simeq \fC_n^2\rtimes \fC_6$ is generated by 
 $$
\mathrm{diag}(\zeta_n, 1, 1), \quad \mathrm{diag}(1,\zeta_n, 1),  \quad \sigma_{123}\cdot \iota. 
    $$     
  \item  $G\simeq (\fC_n\times \fC_{n/r})\rtimes \fC_6$ is generated by 
 $$
\mathrm{diag}(\zeta_n^r, 1, 1), \quad \mathrm{diag}(\zeta_n^s, \zeta_n, 1), \quad \sigma_{123}\cdot \iota, 
    $$       
    where $r\ge 1$, $r\mid n$, and $s^2-s+1\equiv 0 \pmod r$.
\end{enumerate}
\end{prop}

\begin{proof}
When $\nu(G)=\fS_3\times \fC_2$, $G$ is generated by a subgroup $H$, which acts birationally to a transitive but imprimitive action on $\bP^2$, and a Cremona involution on $\bP^2_{x_1,x_2,x_3}$ given by  
$$
\iota_{ab}: (x_1,x_2,x_3)\mapsto (\frac{1}{x_1},\frac{a}{x_2},\frac{b}{x_3}),\quad a,b\in k^\times.
$$
By \cite{DI}, $H$ is generated by $G_T$, $\sigma_{123}$ and $\sigma_{12}$, where $G_T\simeq \fC_n^2$ or $\fC_n\times \fC_{n/3}$. Observe that 
\begin{align}\label{eqn:el1}
(\sigma_{123}^2\cdot\sigma_{12}\cdot\iota_{ab}\cdot\sigma_{12}\cdot\iota_{ab}\cdot\sigma_{123})^{-1}=\mathrm{diag}(a,1,a^2).
\end{align}
    Multiplying $\iota_{ab}$ by \eqref{eqn:el1}, we may assume that $a=1$. 
    Then 
    $$    \sigma_{123}^2\cdot\iota_{ab}\cdot\sigma_{123}\cdot \iota_{ab}=\mathrm{diag}(b^2,b,1).
    $$
    When $G_T=\fC_n^2$, for some $n$, it follows that $\mathrm{diag}(1,1,b)\in G_T$. We may assume $\iota_{ab}=\iota$ is the standard Cremona transformation and we obtain Case (1). 

    When $G_T=\fC_n\times \fC_{n/3}$, for some $n$ divisible by 3, we know that $b$ is a power of $\zeta_n$. Up to multiplying $\iota_{ab}$ by a power of $\mathrm{diag}(1,1,\zeta_n^3)\in G_T$, we may assume that $b=1 $ or $\zeta_3$. These two choices give conjugated subgroups. Indeed, the groups 
    $$
 \langle  \mathrm{diag}(\zeta_n^3, 1, 1), \quad \mathrm{diag}(\zeta_n^2, \zeta_n, 1), \quad \sigma_{123},\quad  \sigma_{12}, \quad \iota\rangle,
    $$
    and 
      $$
   \langle \mathrm{diag}(\zeta_n^3, 1, 1), \quad \mathrm{diag}(\zeta_n^2, \zeta_n, 1), \quad \sigma_{123},\quad  \sigma_{12}, \quad \iota\cdot\mathrm{diag}(1,1,\zeta_3) \rangle
    $$
    are conjugated in $\Aut(S)$ via 
    $$
    (x_1,x_2,x_3)\mapsto (\frac{1}{x_1},\frac{1}{x_2},\frac{\zeta_3^2}{x_3}).
    $$
   We obtain Case (2).

   When $\nu(G)=\fS_3'$, the same argument gives rise to Case (3) and (4). When $\nu(G)=\fC_6$, $G$ is generated by $G_T$ and 
   $$
  \tau: (x_1,x_2,x_3)\mapsto(\frac{a}{x_2}, \frac{b}{x_3},\frac1{x_1}).
   $$
 Up to conjugation by  
 $$
 (x_1,x_2,x_3)\mapsto (\frac ab x_1, ax_2,x_3),
 $$
 we may assume that $a=b=1$, i.e., $\tau=\sigma_{123}\cdot\iota$. By \cite[Theorem 4.7]{DI}, there are two possibilities for $G_T$, giving rise to Case (5) and (6) in the assertion.
\end{proof}

\subsection*{Sarkisov links}

By \cite[Section 7]{DI}, every Sarkisov link starting from $S$ is of type II, i.e., of the form
$$
\xymatrix{
&\tilde{S}\ar[dl]_{\rho}\ar[dr]^{\rho'}&\\
S\ar@{-->}[rr]^{\chi}&&S'
}
$$
where $\rho$ and $\rho'$ are blowups of $d\le 5$ points in general position on $S'$, and one of the following holds:
\begin{itemize}
    \item $d=5$, $S'$ is $G$-biregular to $S$, $\chi$ is the Bertini involution,
    \item $d=4$, $S'$ is $G$-biregular to $S$, $\chi$ is the Geiser involution,
    \item $d=3$, $S'$ is also a dP$_6$, 
      \item $d=2$, $S'$ is also a dP$_6$, 
      \item $d=1$, $S'=\bP^1\times\bP^1$.
\end{itemize}

\subsection*{$G$-Rigidity}
\label{sect:rigid-6}

We introduce special finite subgroups of $\mathrm{Aut}(S)$:
\begin{align*}
G_1&=\langle\mathrm{diag}(-1, 1, 1),\sigma_{123},\sigma_{12},\iota\rangle\simeq(\fC_2)^2\rtimes(\fS_3\times \fC_2)\simeq \fC_2\times\fS_4,\\
G_2&=\langle\mathrm{diag}(-1, 1, 1), \sigma_{123}, \sigma_{12}\cdot\iota\rangle\simeq (\fC_2)^2\rtimes \fS_3\simeq \fS_4,\\
G_3&=\langle\mathrm{diag}(-1, 1, 1), \sigma_{123}\cdot \iota\rangle\simeq (\fC_2)^2\rtimes \fC_6\simeq \fC_2\times\fA_4,\\
G_4&=\langle\mathrm{diag}(\zeta_3,\zeta_3^2,1),\sigma_{123},\sigma_{12},\iota\rangle\simeq \fC_3\rtimes(\fS_3\times \fC_2)\simeq\mathfrak{S}_3^2,\\
G_5&=\langle\mathrm{diag}(\zeta_3,\zeta_3^2,1), \sigma_{123},\sigma_{12}\cdot \iota\rangle\simeq \fC_3\times \fS_3\simeq \fC_3\rtimes \fC_6,\\
G_6&=\langle\mathrm{diag}(\zeta_3,\zeta_3^2,1),\sigma_{123}\cdot\iota\rangle\simeq \fC_3\rtimes \fC_6\simeq \fC_3\times \fS_3,\\
G_7&=\langle\sigma_{123},\sigma_{12},\iota\rangle\simeq \fS_3\times \fC_2,\\
G_8&=\langle\sigma_{123},\sigma_{12}\cdot \iota\rangle\simeq \fS_3,\\
G_9&=\langle\sigma_{123}\cdot \iota\rangle\simeq \fC_6.
\end{align*}
Observe that $G_1$ contains $G_2$ and $G_3$; $G_4$ contains $G_5\simeq G_6$; and $G_7$ contains $G_8$ and $G_9$.

If $G$ is $G_1$, $G_2$, or $G_3$, then $S$ contains a unique $G$-orbit of length $4$. Blowing it up, we obtain the $G$-equivariant commutative diagram:
\begin{align}\label{eqn:etadp6}
    \xymatrix{
\widetilde{S}\ar[d]_{\varpi}\ar[rr]^{\varphi}&&\widetilde{S}\ar[d]^{\varpi}\\
S\ar@{-->}[rr]^{\eta}&&S
}
\end{align}
where $\varpi$ is the blowup of the $G$-orbit of length $4$, $\widetilde{S}$ is the Fermat del Pezzo surface of degree $2$, $\varphi$ is a biregular involution, and $\eta$ is a Geiser birational involution that centralizes $G$.

Similarly, $S$ has a unique $G_4$-orbit of length $3$, whose points are in general position, which is also a $G_5$-orbit and $G_6$-orbit. Moreover, we have the following commutative diagram:
\begin{align}\label{eqn:taudp6}
\xymatrix{
\widehat{S}\ar[d]_{\rho}\ar[rr]^{\phi}&&\widehat{S}\ar[d]^{\rho}\\
S\ar@{-->}[rr]^{\tau}&&S
}
\end{align}
where $\rho$ is the blowup of the orbit of length $3$, $\widehat{S}$ is the Fermat cubic surface, $\phi$ is a biregular involution, and $\tau$ is a birational involution such that 
$$
\tau G_4\tau=G_4,
$$
so $\tau$ normalizes $G_4$, and $\langle G_4,\tau\rangle\simeq\mathfrak{S}_3\wr \fC_2$ has ${\tt GAPID(72,40)}$. It was noticed in \cite{yasinski}, that $\tau$ does not normalize $G_5$ and $G_6$: we have 
$$
\tau G_5\tau=G_6.
$$
In particular, $G_5$ and $G_6$ are conjugate in $\mathrm{Cr}_2$, so $S$ is neither $G_5$-birationally rigid nor $G_6$-birationally rigid.

Recall from \cite{Iskovskikh2003,Iskovskikh2008} that the action of  $G_7$ on $S$ is not linearizable, and $S$ is $G_7$-birational to a conic bundle. Hence,  $S$ is not $G_7$-solid. 
One can show that $\Aut^{G_7}(S)=G_7$, so $\bar{\beta}(\Aut^{G_7}(S))$ is trivial.
Note also that $G_7$ fixes a point in $S$, so blowing up the $G_7$-fixed point, we get a $G_7$-Sarkisov link that ends at $\mathbb{P}^1\times\mathbb{P}^1$, and the $G_7$-action on $\mathbb{P}^1\times\mathbb{P}^1$ has been studied in Lemma~\ref{lemm:D6}, which yields
$$
\bar{\beta}(\Bir^{G_7}(S))=\mathrm{Out}(G_7)\simeq \fC_2.
$$
Finally, we recall from \cite{PSY} that the actions of $G_8$ and $G_9$ on $S$ are linearizable. Indeed, if $G$ is one of these groups, then $G$ fixes a point in $S$, and blowing it up we obtain a $G$-birational map to $\mathbb{P}^1\times\mathbb{P}^1$, with $G$ fixing a point in $\mathbb{P}^1\times\mathbb{P}^1$, so the $G$-action is linearizable.

\begin{theo}
If $G$ is one of the groups in Proposition~\ref{prop:dp6} then: 
\begin{itemize}
\item  The surface $S$ is $G$-birationally superrigid if and only if $G$ is not conjugated to one of $G_1,G_2,G_3,G_4,G_5,G_5,G_7,G_8,G_9$. 

\item If $G$ is $G_1$, $G_2$, or $G_3$, then $S$ is $G$-birationally rigid, 
$\mathrm{Bir}^G(S)$ is generated by $\mathrm{Aut}^G(S)$, $\eta$ given in \eqref{eqn:etadp6}, and 
$$
\bar{\beta}(\Aut^G(S))=\bar{\beta}(\Bir^G(S)).
$$

\item If $G=G_4$, then $S$ is $G$-birationally rigid, $\mathrm{Bir}^G(S)$ is generated by $\mathrm{Aut}^G(S)$ and  $\tau$ given in \eqref{eqn:taudp6}, $\bar{\beta}(\Aut^G(S))$ is trivial, and
$$
\bar{\beta}(\Bir^G(S))=\mathrm{Out}(G)\simeq \fC_2.
$$

\item If $G=G_5$ or $G=G_6$, then 
$$
\mathrm{Bir}^G(S)=\mathrm{Aut}^G(S),
$$
the surface $S$ is $G$-birationally solid, and the only $G$-Mori fibre spaces that are $G$-birational are $S$ with $G_5$ and $G_6$-actions.
 \end{itemize}
 
\end{theo}
\begin{proof}
Using the description of $G$-Sarkisov links that start at $S$, we see that such links do not exists if $|G_T|\geq5$. Note that $|G_T|\ne 2$,$5$, by Proposition~\ref{prop:dp6}. Hence, 
$$
|G_T|\in\{1,3,4\}.
$$
Using Proposition~\ref{prop:dp6} again, we see that $G$ is conjugated to one of the groups 
$G_1,G_2,G_3,G_4,G_5,G_5,G_7,G_8,G_9$, so we may assume that $G$ is one of these. 

If $G$ is $G_1$, $G_2$, or $G_3$, then $S$ has a unique $G$-orbit of length $4$, and this is the only orbit of length $\leq 5$. Using the classification of $G$-Sarkisov links, we see that $S$ is $G$-birationally rigid, and $\mathrm{Bir}^G(S)$ is generated by $\mathrm{Aut}^G(S)$ and the involution $\eta$. 

Similarly, if $G$ is $G_4$, $G_5$ or $G_6$, then  $S$ has a unique $G$-orbit of length $3$, up to conjugation by $\Aut^G(S)$, whose points are in general position. Moreover, $S$ does not have $G$-orbits of length $1$, $2$, $4$ or $5$. Arguing as above, we see that $S$ is $G_4$-birationally rigid,  and $\mathrm{Bir}^{G_4}(S)$ is generated by $\mathrm{Aut}^{G_4}(S)$ and the involution $\tau$. If $G=G_5$ or $G=G_6$, then the only $G$-Sarkisov link that starts at $S$ is given by $\tau$, which is not $G$-birational. This implies that $S$ is $G$-solid, and the only $G$-Mori fiber spaces that are $G$-birational to $S$ are the surface $S$ equipped with $G_5$ and $G_6$-actions. This also shows that $\mathrm{Bir}^G(S)=\mathrm{Aut}^G(S)$ as claimed. 
\end{proof}

\section{Del Pezzo surfaces of degree 4}
\label{sect:dp4}

We follow the presentation in \cite[Section 6.4]{DI}. 

\subsection*{Classification of groups}

There is a natural action of the Weyl group $W(\mathsf D_5)$ on $\bP^4$, via 
the projectivization of a faithful 5-dimensional representation. 
A del Pezzo surface of degree 4 can be given as an intersection of two diagonal quadrics $S:=Q_1\cap Q_2\subset \bP^4$,
see, e.g., \cite[Lemma 6.5]{DI}. We have a diagram

\ 

\centerline{
\xymatrix{
1\ar[r] &  \fC_2^4\ar[r] \ar@{=}[d]& W(\mathsf D_5) \ar[r] & \fS_5 \ar[r] &1 \\
1\ar[r] & \fC_2^4 \ar[r]    & \Aut(S) \ar[r]  \ar@{^{(}->}[u]      &\overline{\Aut(S)} \ar[r]  \ar@{^{(}->}[u] &1 
}
}

\ 

\noindent 
The possibilities for $\overline{\Aut(S)}$ are
$$
\fC_1, \quad \fC_2,\quad \fC_4, \quad \fC_3, \quad \fS_3, \quad \fC_5,\quad  \mathfrak D_5.  
$$
Up to projectivity, the corresponding surfaces $S$ are given by 
\begin{itemize}
    \item[(I)] $\fC_1$:  $\sum\limits_{j=0}^4x_j^2 = \sum\limits_{j=1}^4 a_j x_j^2 =0$, $a_j$ general;
    \item[(II)] $\fC_2$: $\sum\limits_{j=0}^4x_j^2 =x_0^2 +ax_1^2-x_2^2 -ax_3^2 =0$, $a\ne0,\pm1$;
    \item[(III)] $\fC_4$: $\sum\limits_{j=0}^4x_j^2 =x_0^2 +\zeta_4x_1^2-x_2^2 -\zeta_4x_3^2 =0$;
    \item[(IV)] $\fS_3$: $x_0^2+\zeta_3x_1^2+\zeta_3^2x_2^2+x_3^2=x_0^2+\zeta_3^2x_1^2+\zeta_3x_2^2+x_4^2=0$;
    \item[(V)] $\mathfrak D_5$: $\sum\limits_{j=0}^4\zeta_5^jx_j^2 =\sum\limits_{j=0}^4\zeta_5^{4-j}x_j^2 = 0$.
\end{itemize}
All cases with $\rk\Pic(S)^G=1$ are listed in \cite[Theorem 6.9]{DI}; 
we recomputed all possibilities and present the actions, with supplementary information. In each type, the generators of $\Aut(S)$ are given by:
\begin{itemize}
  \item[(I)] $\iota_{i}$, $i=1,\ldots,4$,
  \item[(II)] $\sigma_{(02)(13)}$, $\iota_{i}$, $i=1,\ldots,4$,
  \item[(III)]  $\sigma_{(0123)}$, $\iota_{i}$, $i=1,\ldots,4$,
  \item[(IV)]  $\sigma_{(12)(34)}$,  $\sigma_{(012)}\cdot \tau$, $\iota_{i}$, $i=1,\ldots,4$,
    \item[(V)]  $\sigma_{(01234)}$,  $\sigma_{(14)(23)}$, $\iota_{i}$, $i=1,\ldots,4$,
\end{itemize}
where $\iota_i$ is the sign change of $x_i$, $\sigma_{s}$ is the permutation of variables corresponding to  $s\in\fS_5$, and 
$$
\tau=\mathrm{diag}(1,1,\zeta_3,\zeta_3^2).
$$

\

\begin{longtable}{|c|c|c|c|c|c|c|}
\hline
   Type &$G$ &GapID& Generators&$G$-rigid\\\hline
   I & $\fC_2^2$ & {\tt (4,2)}&$\iota_i, \iota_j$&no\\\hline
   I & $\fC_2^3$ & {\tt (4,2)}&$\iota_i, \iota_j, \iota_r$&yes\\\hline
    I & $\fC_2^4$ & {\tt (4,2)}&$\iota_1,\iota_2,\iota_3,\iota_4$&yes\\\hline
   \hline
II & $\fC_2\times \fC_4$ & {\tt (8,2)}&$\iota_4,\sigma_{(02)(13)}\iota_i,i=0,1$&yes\\\hline
II & $\fD_4$ & {\tt (8,3)}&$\sigma_{(02)(13)}, \iota_i,i=0,1$&no\\\hline
II & $\fC_2^2\rtimes \fC_4$ & {\tt (16,3)}&$\sigma_{(02)(13)}\iota_0, \iota_0\iota_1, \iota_0\iota_2$&yes\\\hline
II & $\fC_2^2\rtimes \fC_4$ & {\tt (16,3)}&${\sigma_{(02)(13)}\iota_0\iota_1, \iota_0\iota_4,\iota_i,i=0,1}$&yes\\\hline
II & $\fC_2\times \fD_4$ & {\tt (16,11)}&$\sigma_{(02)(13)},\iota_4,\iota_0\iota_2,\iota_i,i=0,1$&yes\\\hline
II & $\fC_2^2\wr \fC_2$ & {\tt (32,27)}&$\sigma_{(02)(13)},\iota_1,\iota_2$&yes\\\hline\hline
    III & $\fC_8$ & {\tt (8,1)}&$\sigma_{(0123)}\iota_2$&no\\\hline
III & $\mathrm{OD16}$ & {\tt (16,6})&$\sigma_{(0123)}\iota_2,\iota_{0}\iota_2$&yes\\\hline
III & $\fC_2^3.\fC_4$ & {\tt (32,7)}&$\sigma_{(0123)}\iota_2,\iota_{0}\iota_1$&yes\\\hline
III & $\fC_2\wr \fC_4$ & {\tt (64,32)}&$\sigma_{(0123)},\iota_0$&yes\\\hline\hline
IV& $\fC_2\times \fC_6$ & {\tt (12,5)}&$\sigma_{(012)}\cdot \tau\cdot\iota_4,\iota_3$&no\\\hline
IV& $\fC_2\times \fA_4$ & {\tt (24,13)}&$\iota_3\iota_4,\iota_0\iota_1,\sigma_{(012)}\tau$&yes\\\hline
IV& $\fC_2^2\times \fA_4$ & {\tt (48,49)}&$\iota_3\iota_4,\iota_0\iota_1,\iota_0\iota_3,\sigma_{(012)}\tau$&yes\\\hline
IV & $\fC_3\rtimes \fC_4$ & {\tt (12,1)}&$\sigma_{(012)}\tau,\sigma_{(12)(34)}\iota_4$&no\\\hline
IV & $\fC_3\rtimes \fD_4$ & {\tt (24,8)}&$\sigma_{(012)}\tau,\sigma_{(12)(34)},\iota_4$&no\\\hline
IV & $\fC_2\times \fS_4$ & {\tt (48,48)}&$\sigma_{(012)}\tau,\sigma_{(12)(34)},\iota_3\iota_4,\iota_0\iota_2$&yes\\\hline
IV & $\fA_4\rtimes \fC_4$ & {\tt (48,30)}&$\sigma_{(012)}\tau,\iota_0\iota_2,\sigma_{(12)(34)}\iota_0\iota_3$&yes\\\hline
IV & $\GL_2(\bF_4)$ & {\tt (96,195)}&$\sigma_{(012)}\tau, \sigma_{(12)(34)},\iota_1,\iota_3$&yes\\\hline\hline
V & $\fC_2^4\rtimes \fC_5$ & {\tt (80,49)}&$\sigma_{(01234)}\iota_2,\iota_2\iota_3,\iota_4$&yes\\\hline
V & $\fC_2^4\rtimes \fD_5$ & {\tt (160,234)}&$\sigma_{(01234)},\sigma_{(14)(23)},\iota_1$&yes\\\hline
\end{longtable}

Recall from  \cite{PSY} that $S$ is $G$-birationally rigid if and only if $G$ does not fix points in $S$, i.e., $S^G=\varnothing$. This explains the entries in the last column. When $S^G=\varnothing$, the  classification of $G$-Sarkisov links implies that every $G$-birational map from $S$ to a smooth del Pezzo surface $S^\prime$ can be decomposed into a composition of $G$-biregular maps, Bertini involutions, and Geiser involutions, so, in particular, $S^\prime$ is $G$-birational to $S$, and $\mathrm{Bir}^G(S)$ is generated by $\mathrm{Aut}^G(S)$, Bertini involutions, and Geiser involutions,
which implies 
$$
\bar{\beta}(\Aut^G(S))=\bar{\beta}(\Bir^G(S)).
$$
Similarly, if $S$ is not $G$-birationally rigid, our classification of $G$-actions implies a recent result of Shramov and Trepalin \cite[Corollary 7.5]{shramov-trepalin}, cf. also \cite[Theorem 3.5]{Elagin}:

\begin{theo}
\label{theo:ST}
If there exists a $G$-birational map $S\dasharrow S^\prime$ such that $S^\prime$ is a smooth del Pezzo surface of degree $4$, then $S$ and $S^\prime$ are $G$-biregular.
\end{theo}

However, if $S^G\ne\varnothing$, the group $\Bir^G(S)$ is not always generated by $\Aut^G(S)$, Bertini involutions, and Geiser involutions. The missing generators 
$$
\chi\in\Bir^G(S)
$$ 
can be decomposed into a sequence of three $G$-Sarkisov links, given by the following $G$-equivariant commutative diagram:
\begin{equation}
\label{Vasya}    
\xymatrix{
&&\hat{S}\ar[dl]_{\eta}\ar[dr]^{\eta^{\prime}}&&\\
&\tilde{S}\ar[dl]_{\sigma}\ar[dr]^{\pi}&&\tilde{S}^\prime\ar[dl]_{\pi^\prime}\ar[dr]^{\sigma^\prime}&\\
S\ar@{-->}[rr]^{\phi}\ar@{-->}@/_1.5pc/[rrrr]_{\chi}&&\mathbb{P}^1&&S}
\end{equation}
where $\phi$ is the rational map induced by the linear projection from the embedded tangent space $T_P(S)\subset\mathbb{P}^4$ of a $G$-fixed point $P\in S^G$, and the remaining maps are:
\begin{itemize}
\item $\sigma$ is the blow up of the point $P$,
\item $\pi$ is a conic bundle, 
\item  $\eta$ is a blow up of a $G$-orbit $\tilde{\Sigma}$ such that 
\begin{itemize}
\item[($\diamondsuit$)] each smooth fiber of $\pi$ contains at most one point in $\tilde{\Sigma}$,
\item[($\heartsuit$)] no points of $\tilde{\Sigma}$ are contained in singular fibers of $\pi$,
\end{itemize}
\item 
$\eta^\prime$ is the contraction of the strict transforms of the fibers of $\pi$ that contain points of $\tilde{\Sigma}$, and 
\item 
$\pi^\prime$ is a conic bundle. 
\item $\sigma^\prime$ is the blow up of a $G$-fixed point. 
\end{itemize}
The $G$-birational self-map $\chi$ is determined by a choice of a $G$-fixed point $P\in S^G$ and a $G$-orbit $\tilde{\Sigma}\subset\tilde{S}$ that satisfies both ($\diamondsuit$) and ($\heartsuit$). 
Note that $\tilde{S}$ and $\tilde{S}^\prime$ in \eqref{Vasya} are smooth cubic surfaces, that need not be $G$-biregular. 

\begin{defi}
A map $\chi\in\Bir^G(S)$ as in \eqref{Vasya} will be called  {\em Iskovskikh self-map} of $S$. 
\end{defi}

Thus, the group $\Bir^G(S)$ is generated by $\Aut^G(S)$, Bertini and Geiser involutions, and Iskovskikh self-maps. Note that Iskovskikh self-maps of $S$ may not exist even if $S^G\ne\varnothing$. However, arguing as in Remark~\ref{remark:link-exists}, we obtain a simple criterion for their existence. Namely, if $P$ is a point in $S^G$, and $\phi\colon S\dasharrow\mathbb{P}^1$ is the rational map induced by the linear projection from the embedded tangent space $T_P(S)\subset\mathbb{P}^4$ of $P$, then $\phi$ is $G$-equivariant, so it induces the exact sequence
$$
1\to G_\phi\to G\to G_{\mathbb{P}^1}\to 1,
$$
where $G_{\phi}$ is the kernel of the $G$-action on $\mathbb{P}^1$, and $G_{\mathbb{P}^1}$ is the image of $G$ in $\mathrm{Aut}(\mathbb{P}^1)$. Conditions ($\diamondsuit$) and ($\heartsuit$) imply that $G_{\phi}$ is cyclic, because 
otherwise each smooth fiber of $\pi$ in \eqref{Vasya} contains at least two points of $\tilde{\Sigma}$. Conversely, if $G_{\phi}$ is cyclic, then $\tilde{S}$ in \eqref{Vasya} always contains $G$-orbits that satisfy both ($\diamondsuit$) and ($\heartsuit$). 

\begin{coro}
\label{coro:dP4-C2-C2}
If $G$ is isomorphic to  
$$
\fC_2^2, \quad \fD_4, \quad \fC_2\times \fC_6, \quad \text{ or } \quad \fC_3\rtimes \fD_4
$$
then $\mathrm{Bir}^G(S)$ does not have Iskovskikh self-maps, and is generated by $\mathrm{Aut}^G(S)$, Bertini and Geiser involutions.  In particular, 
$$
\bar{\beta}(\Aut^G(S))=\bar{\beta}(\Bir^G(S)).
$$
\end{coro}

\begin{proof}
For every $P\in S^G$, the kernel of $G\to\mathrm{Aut}(\mathbb{P}^1)$ described above contains a subgroup isomorphic to $\fC_2^2$, so Iskovskikh self-maps of $S$ do not exists.
\end{proof}

\begin{lemm}
\label{lemm:dP4-C8}
If $G\simeq \fC_8$ then 
$$
\fC_2\simeq\bar{\beta}(\Aut^G(S))=\bar{\beta}(\Bir^G(S))\ne\mathrm{Out}(\fC_8)\simeq \fC_2^2.
$$
\end{lemm}

\begin{proof}
Recall that  $S$ is of type III, $G$ is generated by $\sigma_{(0123)}\iota_2$. Then
$$
\Aut^G(S)=\langle G,\iota_0\iota_2\rangle\simeq\mathrm{OD16},
$$
which implies that $\bar{\beta}(\Aut^G(S))$ is a subgroup of order $2$ in $\mathrm{Out}(G)$. Recall that $\mathrm{Out}(\fC_8)\simeq \fC_2^2$. If $\phi\in\mathrm{Out}(G)$ sends $\sigma_{(0123)}\iota_2$ to $(\sigma_{(0123)}\iota_2)^3$  then $\phi\notin\bar{\beta}(\Aut^G(S))$. It suffices  to show that $\phi\notin\bar{\beta}(\Bir^G(S))$. This can be done using the Burnside formalism: Let $g$ be a generator of $G$ and 
consider the actions
$$
\varphi_1,\varphi_2:G\to \Aut(X),
$$
given by 
$$
\varphi_1(g)=\sigma_{(0123)}\iota_2 \quad \text{ and } \quad \varphi_2(g)=(\sigma_{(0123)}\iota_2)^3.
$$ 
Every $G$-action on $X$ is given by either $\varphi_1$ or $\varphi_2$, up to isomorphism. 
In each case, the $G$-action is in standard form: it is free on $X$ away from the elliptic curve 
$$
C:=X\cap\{x_4=0\},
$$
which has generic stabilizer $\fC_2=\langle \iota_4\rangle$
and residual $\fC_4$-action. 
We find incompressible symbols in the Burnside class of each action 
$$
[X\actsfromright\varphi_i(G)]^{\mathrm{inc}}=(\fC_2,\varphi_i'(\fC_4)\actsfromleft k(C),(1))\in \Burn^{\mathrm{inc}}_2(G), \quad i=1,2,
$$
where $\varphi_i': \fC_4\to \Aut(C)$ gives the corresponding residual action on $X$. Let $g'$ be a generator of $\fC_4$, then $\varphi_1'(g')=\varphi_2'(g')^3$. The $\fC_4$-action on $C$ fixes two points. We compute the Burnside classes of the $\fC_4$-action on $C$ and obtain 
$$
[C\actsfromright\varphi_1'(\fC_4)]=2(\fC_4,\mathrm{triv}\actsfromleft k, (1))+(\fC_2,\mathrm{triv}\actsfromleft k, (1))\in\Burn_1(\fC_4),
$$
$$
[C\actsfromright\varphi_2'(\fC_4)]=2(\fC_4,\mathrm{triv}\actsfromleft k, (3))+(\fC_2,\mathrm{triv}\actsfromleft k, (1))\in\Burn_1(\fC_4).
$$
Since there is no relation in $\Burn_1(\fC_4)$, we see that 
$$
[C\actsfromright\varphi_1'(\fC_4)]\ne [C\actsfromright\varphi_2'(\fC_4)],
$$
which implies that the actions of $\fC_4$ on $C$ given by $\varphi_1'$ and $\varphi_2'$ are not equivariantly birational. It follows that 
$$
[X \actsfromright \varphi_1(G)]\ne[X \actsfromright \varphi_2(G)].
$$
\end{proof}

\begin{rema}
One can also prove Lemma~\ref{lemm:dP4-C8} by describing Iskovskikh self-maps in $\Bir^G(S)$ and their images in $\bar{\beta}(\Bir^G(S))\subset\mathrm{Out}(G)$. In this case,  $G$ fixes two points in $S$, which are swapped by $\Aut^G(S)\simeq\mathrm{OD16}$, so we let $P$ be one of these points, and $Q$ the other. Consider the diagram \eqref{Vasya} for $\sigma$ being the blowup of $P$ and some $G$-orbit $\tilde{\Sigma}\subset\tilde{S}$. Then $G_\phi\simeq\mathfrak{C}_2$, where $G_\phi$ is the kernel of the $G$-action on $\mathbb{P}^1$ induced by $\phi$ in \eqref{Vasya} and ($\diamondsuit$) is equivalent to 
$$
\tilde{\Sigma}\subset \widetilde{S}^G_\phi,
$$
where $\widetilde{S}^G_\phi$ is the strict transform of the curve $\{x_4\}\subset S$. Then either $\tilde{\Sigma}$ is the $G$-fixed point that is mapped to $Q$, or $|\Sigma|=4$ (general case). In the former case, the Iskovskikh self-map $\chi$ is the $\Aut^G(S)$-equivariant Geiser involution induced by blowing up $S$ at $P$ and $Q$, so it centralizes $G$, and lies in the kernel of $\bar{\beta}\colon \Bir^G(S)\to\mathrm{Out}(G)$. Similarly, we can describe $\chi$ in the case when $|\Sigma|=4$ and show that its image in $\mathrm{Out}(G)$ is also trivial. 
\end{rema}

\begin{lemm}
\label{lemm:dP4-C3C4}
If $G\simeq \fC_3\rtimes \fC_4$ then 
$$
\bar{\beta}(\Aut^G(S))=\bar{\beta}(\Bir^G(S))=\mathrm{Out}(\fC_3\rtimes\fC_4)\simeq \fC_2.
$$
\end{lemm}
\begin{proof}
We have $\Aut^G(S)=\fC_3\rtimes\fD_4$ and $\Out(\fC_3\rtimes\fC_4)=\fC_2$. One can check that $\bar\alpha: \Aut^G(S)\to \Out(\fC_3\rtimes\fC_4)$ is surjective, and so is $\bar \beta$. 
\end{proof}
\appendix

\section{Tables}
\label{sect:tables}

\subsection*{Cubic surfaces}

Let $S\subset \bP^3$ be a smooth cubic surface and $G\subseteq\mathrm{Aut}(S)$ a subgroup such that $\rk\Pic(S)^G=1$. Then  $S$ is one of the following surfaces:

\begin{itemize}
\item[(I)] $\sum_{i=1}^4x_1^3=0$;
\item[(II)]$(\sum_{i=1}^4x_i^2)(\sum_{i=1}^4x_i)+\displaystyle{\sum_{1\leq i<j<k\leq4 }2x_ix_jx_k}=\sum_{i=1}^4x_i^3$;
\item[(III)]$x_1^3+x_2^3+x_3^3+x_4^3+6ax_2x_3x_4=0$, where $a=\frac{-1+\sqrt3}{2}$;
\item[(IV)]$x_1^3+x_2^3+x_3^3+x_4^3+6ax_2x_3x_4=0$, where $a\in k$ is general;
\item[(V)] $x_1^3+x_1(x_2^2+x_3^2+x_4^2)+ax_2x_3x_4=0$, where  $a\in k$ is general;
\item[(VI)]$x_1^3+x_2^3+x_3^3+x_4^3+ax_3x_4(x_1+x_2)=0$, where  $a\in k$ is general;
\item[(VIII)]$x_3^3 +x_4^3 +ax_3x_4(x_1 +bx_2)+x_1^3 +x_2^3=0$, where  $a,b\in k$ are general.
\end{itemize}

\

We introduce the actions: $\sigma_s$ is the permutation of coordinates via $s\in\fS_4$ and 
\begin{align*}
\tau_1&:\mathrm{diag}(1,1,\zeta_3^2,\zeta_3),\\
\tau_2&:\mathrm{diag}(1,-1,-1,1),\\
\tau_3&:\mathrm{diag}(\zeta_3,1,1,1),\\
\tau_4&:\mathrm{diag}(1,\zeta_3^2,\zeta_3,1),\\
\tau_5&:\mathrm{diag}(\zeta_3,\zeta_3^2,1,1),\\
\tau_6&:\mathrm{diag}(1,1,1,\zeta_3),\\
\tau_7&:\mathrm{diag}(\zeta_3,1,\zeta_3,1),\\
\alpha_1&:\mathsf{(x)}\mapsto(-x_1-x_2-x_3-x_4,x_1,x_2,x_3),\\
\alpha_2&:\mathsf{(x)}\mapsto{\sf(x)}\begin{pmatrix}
    b&0&0&0\\
    0&1&1&1\\
    0&1&\zeta_3&\zeta_3^2\\
    0&1&\zeta_3^2&\zeta_3
\end{pmatrix},\quad b^3=3\sqrt3. \\
\end{align*}

The table below lists all cases when $\rk\Pic(S)^G=1$.
\begin{longtable}{|c|c|c|c|c|c|c|}
\hline
   Type &$G$ &GapID& Generators\\\hline
I & $\fC_3$ & {\tt (3,1)}&$\tau_3$\\\hline
I & $\fS_3$ & {\tt (6,1)}&$\sigma_{(23)},\sigma_{(234)}$\\\hline
I & $\fC_6$ & {\tt (6,2)}&$\tau_3\sigma_{(34)}$\\\hline
I & $\fC_6$ & {\tt (6,2)}&$\tau_5\sigma_{(34)}$\\\hline
I & $\fS_3$ & {\tt (6,1)}&$\tau_1,\sigma_{(34)}$\\\hline
I & $\fC_3^2$ & {\tt (9,2)}&$\tau_1,\tau_3$\\\hline
I & $\fC_3^2$ & {\tt (9,2)}&$\tau_1,\tau_6$\\\hline
I & $\fC_3^2$ & {\tt (9,2)}&$\tau_1^2\sigma_{(234)},\tau_1\tau_4\sigma_{(234)}$\\\hline
I & $\fC_9$ & {\tt (9,1)}&$\tau_1\tau_3^2\tau_5\sigma_{(234)}$\\\hline
I & $\fD_6$ & {\tt (12,4)}&$\sigma_{(34)},\tau_1,\sigma_{(12)(34)}$\\\hline
I & $\fC_3\times\fS_3$ & {\tt (18,3)}&$\sigma_{(34)},\tau_1,\tau_{3}\tau_{6}$\\\hline
I & $\fC_3\times\fS_3$ & {\tt (18,3)}&$\tau_1,\tau_6,\sigma_{(34)}$\\\hline
I & $\fC_3\times\fS_3$ & {\tt (18,3)}&$\tau_1,\tau_3,\sigma_{(34)}$\\\hline
I & $\fC_3\times\fS_3$ & {\tt (18,3)}&$\tau_6,\tau_1\tau_4,\sigma_{(13)(24)}$\\\hline
I & $\fC_3\times \fC_6$ & {\tt (18,5)}&$\tau_3\tau_5^2,\tau_3\sigma_{(34)}$\\\hline
I & $\fC_3\times\fS_3$ & {\tt (18,3)}&$\tau_3,\sigma_{(34)},\tau_1\sigma_{(234)}$\\\hline
I & $\fC_3^3$ & {\tt (27,5)}&$\tau_1,\tau_3,\tau_6$\\\hline
I & $\mathrm{He_3}$ & {\tt (27,3)}&$\tau_1,\tau_3,\sigma_{(234)}$\\\hline
I & $\fC_9:\fC_3$ & {\tt (27,4)}&$\tau_1,\tau_3\tau_4\tau_5\tau_6^2\sigma_{(234)}$\\\hline
I & $\fS_4$ & {\tt (24,12)}&$\sigma_{(12)},\sigma_{(1234)}$\\\hline
I & $\fS_3^2$ & {\tt (36,10)}&$\tau_1,\tau_7,\sigma_{(12)},\sigma_{(34)}$\\\hline
I & $\fC_6\times\fS_3$ & {\tt (36,12)}&$\tau_1,\tau_3\tau_5,\sigma_{(12)},\sigma_{(34)}$\\\hline
I & $\fC_3\times \fC_3:\fS_3$ & {\tt (54,13)}&$\tau_1,\tau_4,\tau_7,\sigma_{(13)(24)}$\\\hline
I & $\fC_3^2:\fS_3$ & {\tt (54,8)}&$\tau_1,\tau_3,\sigma_{(23)},\sigma_{(234)}$\\\hline
I & $\fC_3^2\times\fS_3$ & {\tt (54,12)}&$\tau_1,\tau_3,\tau_6,\sigma_{(34)}$\\\hline
I & $\fC_3\wr \fC_3$ & {\tt (81,7)}&$\tau_1,\tau_3,\tau_6,\sigma_{(234)}$\\\hline
I & $\fS_3\wr \fC_2$ & {\tt (72,40)}&$\tau_1,\tau_7,\sigma_{(12)},\sigma_{(34)},\sigma_{(14)(23)}$\\\hline
I & $\fC_3^3:\fC_2^2$ & {\tt (108,40)}&$\tau_1,\tau_4,\tau_7,\sigma_{(12)(34)},\sigma_{(14)(23)}$\\\hline
I & $\fC_3^2:(\fC_3:\fC_4)$ & {\tt (108,37)}&$\tau_1,\tau_4,\tau_7,\sigma_{(12)(34)},\sigma_{(1324)}$\\\hline
I & $\fC_3\times\fS_3^2$ & {\tt (108,38)}&$\tau_1,\tau_4,\tau_7,\sigma_{(12)},\sigma_{(34)}$\\\hline
I & $\fC_3\wr \fS_3$ & {\tt (162,10)}&$\tau_1,\tau_3^2\tau_6,\tau_1\tau_4,\sigma_{(3,4)},\sigma_{(234)}$\\\hline
I & $\fS_3^2:\fS_3$ & {\tt (216,158)}&$\tau_3\tau_7\sigma_{(34)},\sigma_{(14)(23)}$\\\hline
I & $\fC_3^3:\fC_2^2:\fC_3$ & {\tt (324,160)}&$\tau_3\tau_7\sigma_{(234)},\sigma_{(14)(23)}$\\\hline
I & $\fC_3^3.\fS_4$ & {\tt (648,704)}&$\tau_3,\sigma_{(12)},\sigma_{(1234)}$\\\hline\hline
     II&$\fS_3$ &{\tt (6,1)}&$\sigma_{(23)},\sigma_{(234)}$\\\hline
      II&$\fC_6$ &{\tt (6,2)}&$\sigma_{14}\alpha_1$\\\hline
       II&$\fD_6$ &{\tt (12,4)}&$\sigma_{(23)},\sigma_{(234)},\sigma_{14}\alpha_1$\\\hline
     II&$\fS_4$ &{\tt (24,12)}&$\sigma_{1234},\sigma_{(12)}$\\\hline
     II&$\fS_5$ &{\tt (120,34)}&$\alpha_1,\sigma_{(12)}$\\\hline
      III&$\fC_3$ &{\tt (3,1)}&$\tau_3$\\\hline
    III&$\fS_3$ &{\tt (6,1)}&$\sigma_{(34)},\tau_1$\\\hline
     III&$\fS_3$ &{\tt (6,1)}&$\tau_1\sigma_{(34)},\tau_1\sigma_{(234)}$\\\hline
        III&$\fC_6$ &{\tt (6,2)}&$\tau_1\tau_3\sigma_{(34)}$\\\hline
       III&$\fC_3^2$ &{\tt (9,2)}&$\tau_3,\tau_1$\\\hline
       III&$\fC_3^2$ &{\tt (9,2)}&$\tau_3,\tau_1\sigma_{(234)}$\\\hline
   III&$\fC_{12}$ &{\tt (12,2)}&$\alpha_2$\\\hline
   III&$\fC_3\times\fS_3$ &{\tt (18,3)}&$\tau_1,\tau_3,\sigma_{(34)}$\\\hline
    III&$\fC_3\times\fS_3$ &{\tt (18,3)}&$\tau_3,\tau_1^2\sigma_{(234)},\tau_1\sigma_{(34)}$\\\hline
 III&$\mathrm{He}_3$ &{\tt (27,3)}&$\tau_1,\tau_3,\sigma_{(234)}$\\\hline
     III&$\fC_3^2\rtimes\fS_3$ &{\tt (54,8)}&$\tau_1,\tau_3,\sigma_{(234)},\sigma_{(23)}$\\\hline
     III&$\fC_3^2\rtimes\fS_3.\fC_2$ &{\tt (108,15)}&$\tau_1,\tau_3,\sigma_{(234)},\sigma_{(23)},\alpha_2$\\\hline\hline
      IV&$\fC_3$ &{\tt (3,1)}&$\tau_3$\\\hline
    IV&$\fS_3$ &{\tt (6,1)}&$\sigma_{(34)}, \tau_1$\\\hline
      IV&$\fS_3$ &{\tt (6,1)}&$\sigma_{(34)},\sigma_{(234)}$\\\hline
      IV&$\fS_3$ &{\tt (6,1)}&$
    \tau_1\sigma_{(34)},\tau_1\sigma_{(234)}$\\\hline
     IV&$\fS_3$ &{\tt (6,1)}&$\tau_1^2\sigma_{(34)},\tau_1^2\sigma_{(234)}$\\\hline
        IV&$\fC_6$ &{\tt (6,2)}&$\tau_1\tau_3\sigma_{(34)}$\\\hline
    IV&$\fC_3^2$ &{\tt (9,2)}&$\tau_3,\tau_1$\\\hline
        IV&$\fC_3^2$ &{\tt (9,2)}&$\tau_3,\tau_1\sigma_{(234)}$\\\hline
       IV&$\fC_3^2$ &{\tt (9,2)}&$\tau_4^2\sigma_{(234)},\tau_1^2\sigma_{(234)}$\\\hline
       IV&$\fC_3^2$ &{\tt (9,2)}&$\tau_3,\sigma_{(234)}$\\\hline
   IV&$\fC_3\times\fS_3$ &{\tt (18,3)}&$\tau_1,\tau_3,\sigma_{(34)}$\\\hline
 IV&$\fC_3\times\fS_3$ &{\tt (18,3)}&$\tau_3,\sigma_{(234)},\tau_1\sigma_{(34)}$\\\hline
      IV&$\fC_3\times\fS_3$ &{\tt (18,3)}&$\tau_3,\tau_1^2\sigma_{(234)},\tau_1\sigma_{(34)}$\\\hline
  IV&$\fC_3\times\fS_3$ &{\tt (18,3)}&$\tau_3,\tau_4\sigma_{(234)},\tau_1\sigma_{(34)}$\\\hline
   IV&$\mathrm{He}_3$ &{\tt (27,3)}&$\tau_1,\tau_3,\sigma_{(234)}$\\\hline
     IV&$\fC_3^2\rtimes\fS_3$ &{\tt (54,8)}&$\tau_1,\tau_3,\sigma_{(234)},\sigma_{(23)}$\\\hline\hline
       V&$\fS_3$&{\tt (6,1)}&$\sigma_{(234)}, \sigma_{(23)}$\\\hline
    V&$\fS_4$ &{\tt (24,12)}&$\sigma_{(234)}, \sigma_{(23)},\tau_2$\\\hline\hline
       VI&$\fC_6$ &{\tt (6,2)}&$\sigma_{(12)}\tau_1$\\\hline
       VI&$\fS_3$ &{\tt (6,1)}&$\sigma_{(34)}, \tau_1$\\\hline
          VI&$\fD_6$ &{\tt (12,4)}&$\sigma_{(12)}\tau_1,\sigma_{(23)}$\\\hline\hline
     VIII&$\fS_3$ &{\tt (6,1)}&$\sigma_{(34)}, \tau_1$\\\hline
\end{longtable}

\subsection*{Del Pezzo surfaces of degree 2}

Let $S\subset\bP(2_{x_0},1_{x_1},1_{x_2},1_{x_3})$ be a del Pezzo surface of degree 2 and $G\subseteq\Aut(S)$ a subgroup such that $\rk\Pic(S)^G=1$. Then $S$ is one of the following surfaces:
\begin{itemize}
    \item[(I)] $x_0^2=x_1^3x_2 + x_1x_3^3 + x_2^3x_3$;
    \item[(II)]$x_0^2=x_1^4+x_2^4+x_3^4$;
    \item[(III)] $x_0^2=x_1^4+x_2^4+x_3^4+(4\zeta_3+2)x_1^2x_2^2$;
    \item[(IV)] $x_0^2=x_1^4+x_2^4+x_3^4-a(x_1^2x_2^2+x_1^2x_3^2+x_2^2x_3^2)$, where $a\in k$ is general;
    \item[(V)] $x_0^2=x_1^4+x_2^4+x_3^4+ax_1^2x_2^2$, where $a\in k$ is general;
    \item[(VII)] $x_0^2= x_1^4 + x_2^4 + x_3^4 + ax_1^2x_2^2 + bx_3^2x_1x_2$, where $a,b\in k$ are general;
    \item[(VIII)]$x_0^2=x_3^3x_1 + x_1^4 + x_2^4 + ax_1^2x_2^2$, where $a\in k$ is general.
\end{itemize}

Put
\begin{align*}
\iota&:(x_0,x_1,x_2,x_3)\mapsto (-x_0,x_1,x_2,x_3),\\
    \sigma_1&:(x_0,x_1,x_2,x_3)\mapsto (x_0,x_1,x_2,x_3)\cdot\begin{pmatrix}
   1&0&0&0\\
   0&a_1&a_2&a_3\\
   0&a_4&a_5&a_6\\
   0&a_7&a_8&a_9
\end{pmatrix},\\
\sigma_2&:(x_0,x_1,x_2,x_3)\mapsto(x_0,x_1,x_2,x_3)\cdot\begin{pmatrix}
   1&0&0&0\\
   0&b_1&b_2&b_3\\
   0&b_4&b_5&b_6\\
   0&b_7&b_8&b_9
\end{pmatrix},\\
\sigma_3&:(x_0,x_1,x_2,x_3)\mapsto (x_0,x_2,x_3,x_1),\\
\sigma_4&:(x_0,x_1,x_2,x_3)\mapsto (x_0,-\zeta_4x_1,\zeta_4x_3,x_2),\\
\sigma_5&:(x_0,x_1,x_2,x_3)\mapsto (x_0,x_1,x_2,x_3)\cdot\begin{pmatrix}
    \zeta_6&0&0&0\\
    0&\frac{1+\zeta_4}{2}&\frac{-1+\zeta_4}{2}&0\\
    0&\frac{1+\zeta_4}{2}&\frac{1-\zeta_4}{2}&0\\
    0&0&0&\zeta_3
\end{pmatrix},\\
\sigma_6&:(x_0,x_1,x_2,x_3)\mapsto (x_0,x_1,x_2,x_3)\cdot\begin{pmatrix}
    \zeta_3&0&0&0\\
    0&\frac{1+\zeta_4}{2}&\frac{-1-\zeta_4}{2}&0\\
    0&\frac{-1+\zeta_4}{2}&\frac{-1+
\zeta_4}{2}&0\\
    0&0&0&\zeta_3^2
\end{pmatrix},\\
\sigma_7&:(x_0,x_1,x_2,x_3)\mapsto (x_0,x_2,-x_1,x_3),\\
\sigma_8&:(x_0,x_1,x_2,x_3)\mapsto (x_0,x_1,x_3,x_2),\\
\sigma_9&:(x_0,x_1,x_2,x_3)\mapsto (x_0,x_1,\zeta_4x_2,\zeta_4^3x_3),\\
\sigma_{10}&:(x_0,x_1,x_2,x_3)\mapsto (x_0,-x_1,x_2,x_3),\\
\sigma_{11}&:(x_0,x_1,x_2,x_3)\mapsto (x_0,\zeta_4x_1,\zeta_4^3x_2,x_3),\\
\sigma_{12}&:(-x_0,x_1,-x_2,\zeta_3x_3)\mapsto (x_0,x_1,\zeta_3x_2,-x_3),
\end{align*}
where 
\begin{align*}
   a_1&=\frac17(3\zeta_7^5 + \zeta_7^4 + \zeta_7^3 + 3\zeta_7^2 - 1),\quad
a_2=\frac17(-\zeta_7^5 + 2\zeta_7^4 + 2\zeta_7^3 - \zeta_7^2 - 2),\\
    a_3&=\frac17(-2\zeta_7^5 - 3\zeta_7^4 - 3\zeta_7^3 - 2\zeta_7^2 - 4),\quad
    a_4=\frac17(-\zeta_7^5 + 2\zeta_7^4 + 2\zeta_7^3 - \zeta_7^2 - 2),\\
    a_5&=\frac17(-2\zeta_7^5 - 3\zeta_7^4 - 3\zeta_7^3 - 2\zeta_7^2 - 4),\quad
    a_6=\frac17(3\zeta_7^5 + \zeta_7^4 + \zeta_7^3 + 3\zeta_7^2 - 1),\\
    a_7&=\frac17(-2\zeta_7^5 - 3\zeta_7^4 - 3\zeta_7^3 - 2\zeta_7^2 - 4),\quad
    a_8=\frac17
(3\zeta_7^5 + \zeta_7^4 + \zeta_7^3 + 3\zeta_7^2 - 1),\\
    a_9&
=\frac17(-\zeta_7^5 + 2\zeta_7^4 + 2\zeta_7^3 - \zeta_7^2 - 2),\quad 
    b_1=\frac17(2\zeta_7^4 - \zeta_7^3 - 2\zeta_7^2 - \zeta_7 + 2),\\
    b_2&=\frac17(-2\zeta_7^5 - \zeta_7^3 + 2\zeta_7^2 + 2\zeta_7 - 1),\quad
    b_3=\frac17(\zeta_7^5 + 3\zeta_7^4 - \zeta_7^3 + 3\zeta_7^2 + \zeta_7),\\
    b_4&=\frac17(-2\zeta_7^5 - \zeta_7^3 + 2\zeta_7^2 + 2\zeta_7 - 1),\quad
    b_5=\frac17(-\zeta_7^5 - \zeta_7^4 + 2\zeta_7^2 - 2\zeta_7 + 2),\\
    b_6&=\frac17(\zeta_7^5 + 4\zeta_7^4 + 2\zeta_7^3 + 2\zeta_7^2 + 4\zeta_7 + 1),\quad
    b_7=\frac17(\zeta_7^5 + 3\zeta_7^4 - \zeta_7^3 + 3\zeta_7^2 + \zeta_7),\\
    b_8&=\frac17(\zeta_7^5 + 4\zeta_7^4 + 2\zeta_7^3 + 2\zeta_7^2 + 4\zeta_7 + 1),\quad
    b_9=\frac17(\zeta_7^5 - \zeta_7^4 + \zeta_7^3 + 3\zeta_7 + 3).
\end{align*}

The table below lists all cases when $\rk\Pic(S)^G=1$.
\begin{longtable}{|c|c|c|c|}
\hline
   Type &$G$ &GapID& Generators\\\hline
I&$\fD_4$&{\tt(8,3)}&$ \sigma_2^2\sigma_1\sigma_2\sigma_1\sigma_2^2, \sigma_2^3\sigma_1\sigma_2\sigma_1\sigma_2^3
\sigma_1\sigma_2 $\\\hline
I&$\fS_4$&{\tt(24,12)}&$ \sigma_2^3\sigma_1\sigma_2, \sigma_2^2\sigma_1\sigma_2\sigma_1\sigma_2^2 $\\\hline
I&$\fS_4$&{\tt(24,12)}&$ \sigma_1\sigma_2\sigma_1, \sigma_2^3\sigma_1\sigma_2 $\\\hline
I&$\PSL_2(\bF_7)$&{\tt(168,42)}&$ \sigma_2, \sigma_1 $\\\hline\hline
II&$\fC_4$&{\tt(4,1)}&$ \iota\sigma_3\sigma_4^{-1}\sigma_3\sigma_4^{-1} $\\\hline
II&$\fC_2\times\fC_4$&{\tt(8,2)}&$ \sigma_4^4, \iota\sigma_3^{-1}\sigma_4^{-2}\sigma_3 $\\\hline
II&$\fC_2\times\fC_4$&{\tt(8,2)}&$ \iota\sigma_3^{-1}\sigma_4^{-2}\sigma_3, \iota\sigma_4^4 $\\\hline
II&$Q8$&{\tt(8,4)}&$ \sigma_4^3\sigma_3^{-1}, \sigma_4\sigma_3^{-1}\sigma_4^2 $\\\hline
II&$\fD_4$&{\tt(8,3)}&$ \sigma_3\sigma_4, \sigma_4\sigma_3^{-1}\sigma_4^{-2} $\\\hline
II&$\fC_2\times\fC_4$&{\tt(8,2)}&$ \sigma_4^3\sigma_3^{-1}, \iota\sigma_3^{-1}\sigma_4^{-2}\sigma_3 $\\\hline
II&$\fD_4$&{\tt(8,3)}&$ \sigma_3\sigma_4, \sigma_3\sigma_4^{-3} $\\\hline
II&$\fC_2\times\fC_4$&{\tt(8,2)}&$ \sigma_3\sigma_4, \iota\sigma_3^{-1}\sigma_4^{-2}\sigma_3 $\\\hline
II&$OD16$&{\tt(16,6)}&$ \sigma_3\sigma_4^{-1}, \sigma_3\sigma_4^3 $\\\hline
II&$\fC_4^2$&{\tt(16,2)}&$ \iota\sigma_4^2, (\sigma_3^{-1}\sigma_4^{-1}\sigma_3)^2 $\\\hline
II&$\fD_4\rtimes\fC_2$&{\tt(16,13)}&$ \sigma_3\sigma_4, \sigma_4\sigma_3^{-1}\sigma_4^{-2}, \iota\sigma_4^3\sigma_3^{-1} $\\\hline
II&$\fD_4\rtimes\fC_2$&{\tt(16,13)}&$ \iota\sigma_3\sigma_4, \sigma_4^3\sigma_3^{-1}, \sigma_4^{-2}\sigma_3\sigma_4^{-1} $\\\hline
II&$\fD_4\rtimes\fC_2$&{\tt(16,13)}&$ \sigma_3\sigma_4, (\sigma_3\sigma_4^{-1})^2, \sigma_4\sigma_3^{-1}\sigma_4^{-2} $\\\hline
II&$\fD_4\rtimes\fC_2$&{\tt(16,13)}&$ \sigma_3\sigma_4, \sigma_4^4, \iota\sigma_3^{-1}\sigma_4^{-2}\sigma_3 $\\\hline
II&$\fS_4$&{\tt(24,12)}&$ \sigma_3^{-1}\sigma_4^{-1}, \sigma_3^{-1}\sigma_4^2\sigma_3^{-1} $\\\hline
II&$\fC_4\wr\fC_2$&{\tt(32,11)}&$ \sigma_3\sigma_4, \iota\sigma_4^2 $\\\hline
II&$\fC_4\wr\fC_2$&{\tt(32,11)}&$ \sigma_3\sigma_4, \sigma_3\sigma_4^{-1} $\\\hline
II&$\fC_4\wr\fC_2$&{\tt(32,11)}&$ \sigma_3\sigma_4^{-1}, \iota\sigma_3\sigma_4 $\\\hline
II&$\fC_4^2\rtimes\fC_3\rtimes\fC_2$&{\tt(96,64)}&$ \sigma_3, \sigma_4^{-1} $\\\hline\hline
III&$\fC_4$&{\tt(4,1)}&$ \iota\sigma_6^{-3} $\\\hline
III&$\fC_6$&{\tt(6,2)}&$ \iota\sigma_6^{-2} $\\\hline
III&$Q8$&{\tt(8,4)}&$ \iota\sigma_6^2\sigma_5^{-1}, \iota\sigma_5^{-1}\sigma_6^2 $\\\hline
III&$\fC_2\times\fC_4$&{\tt(8,2)}&$ \sigma_5\sigma_6^{-1}\sigma_5, \sigma_6\sigma_5^{-1}\sigma_6 $\\\hline
III&$\fC_2\times\fC_4$&{\tt(8,2)}&$ \iota\sigma_6^{-1}\sigma_5\sigma_6^{-1}, \sigma_5^2\sigma_6\sigma_5^{-1} $\\\hline
III&$\fD_4$&{\tt(8,3)}&$ \sigma_5^2\sigma_6^{-1}, \sigma_5^{-2}\sigma_6 $\\\hline
III&$\fC_{12}$&{\tt(12,2)}&$ \iota\sigma_6^{-1} $\\\hline
III&$\fD_4\rtimes\fC_2$&{\tt(16,13)}&$ \sigma_5^{-1}\sigma_6^{-1}, \sigma_5\sigma_6, \iota\sigma_6^2\sigma_5^{-1} $\\\hline
III&$\fD_4\rtimes\fC_2$&{\tt(16,13)}&$ \sigma_6^{-3}, \sigma_5^2\sigma_6^{-1}, \sigma_5^{-2}\sigma_6 $\\\hline
III&$\fD_4\rtimes\fC_2$&{\tt(16,13)}&$ \sigma_6^{-2}\sigma_5, \iota\sigma_5^{-1}\sigma_6^{-1}, \sigma_6^2\sigma_5^{-1} $\\\hline
III&$\SL_2(\bF_3)$&{\tt(24,3)}&$ \iota\sigma_5^{-1}, \sigma_6^2 $\\\hline
III&$\SL_2(\bF_3)\rtimes\fC_2$&{\tt(48,33)}&$ \sigma_6, \iota\sigma_5 $\\\hline
III&$\SL_2(\bF_3)\rtimes\fC_2$&{\tt(48,33)}&$ \sigma_5^{-1}\sigma_6, \sigma_5^{-1}\sigma_6^{-1} $\\\hline\hline
IV&$\fD_4$&{\tt(8,3)}&$ \sigma_3  \sigma_7, \sigma_3  \sigma_7^{-1} $\\\hline
IV&$\fS_4$&{\tt(24,12)}&$\sigma_3, \sigma_7 $\\\hline\hline
V&$\fC_4$&{\tt(4,1)}&$ \sigma_{10}\iota\sigma_{11} $\\\hline
V&$\fD_4$&{\tt(8,3)}&$ \sigma_{11}^{-1}, \sigma_3^{-1}\sigma_{10} $\\\hline
V&$Q8$&{\tt(8,4)}&$ \sigma_3^{-1}, \sigma_{11}^{-1} $\\\hline
V&$\fD_4$&{\tt(8,3)}&$ \sigma_3^{-1}, \sigma_{10} $\\\hline
V&$\fC_2\times\fC_4$&{\tt(8,2)}&$ \sigma_3\iota, \sigma_{10}\iota\sigma_{11}^{-1} $\\\hline
V&$\fD_4$&{\tt(8,3)}&$ \sigma_{10}, \sigma_3\sigma_{11} $\\\hline
V&$\fC_2\times\fC_4$&{\tt(8,2)}&$ \sigma_{11}^{-1}, \sigma_{10}\iota $\\\hline
V&$\fC_2\times\fC_4$&{\tt(8,2)}&$ \sigma_3^{-1}, \sigma_{10}\iota\sigma_{11}^{-1} $\\\hline
V&$\fC_2\times\fC_4$&{\tt(8,2)}&$ \sigma_3\sigma_{11}^{-1}, \sigma_{10}\iota\sigma_{11}^{-1} $\\\hline
V&$\fC_2\times\fC_4$&{\tt(8,2)}&$ \sigma_3\sigma_{10}, \sigma_{10}\iota\sigma_{11}^{-1} $\\\hline
V&$\fC_2\times\fC_4$&{\tt(8,2)}&$ \sigma_{10}, \iota\sigma_{11} $\\\hline
V&$\fD_4\rtimes\fC_2$&{\tt(16,13)}&$ \sigma_{11}^{-1}, \sigma_{10}, \sigma_3^{-1} $\\\hline
V&$\fD_4\rtimes\fC_2$&{\tt(16,13)}&$ \sigma_{11}, \sigma_{10}\iota, \sigma_3^{-1}\sigma_{10} $\\\hline
V&$\fD_4\rtimes\fC_2$&{\tt(16,13)}&$ \sigma_3, \sigma_{11}^{-1}, \sigma_{10}\iota $\\\hline
V&$\fD_4\rtimes\fC_2$&{\tt(16,13)}&$ \sigma_{10}, \sigma_3\iota, \sigma_3\sigma_{11}^{-1} $\\\hline
V&$\fD_4\rtimes\fC_2$&{\tt(16,13)}&$ \sigma_{10}, \sigma_3^{-1}, \iota\sigma_{11}^{-1} $
\\\hline\hline
VII&$\fD_4$&{\tt(8,3)}&$ \sigma_8,\sigma_9$\\\hline\hline
VIII&$\fC_6$&{\tt(6,2)}&$ \sigma_{12}$\\\hline
 \end{longtable}

\subsection*{Del Pezzo surfaces of degree $1$} 
Let $S\subset\bP(1_{x_0},1_{x_1},2_{x_2},3_{x_3})$ be a del Pezzo surface of degree $1$. Then  
$$
S=\{x_3^2 + x_2^3 + F_4(x_0, x_1)x_2 + F_6(x_0, x_1) = 0\},
$$
where $F_4$ and $F_6$ are homogeneous polynomials of degree $4$ and $6$, respectively. Put 
\begin{align*}
  \gamma&: (x_0,x_1,x_2,x_3)\mapsto (x_0,x_1,x_2,-x_3),
\\
\delta&: (x_0,x_1,x_2,x_3)\mapsto (x_0,x_1,\zeta_3x_2,x_3). 
\end{align*}

\begin{lemm}
    Let $S$ be a del Pezzo surface of degree 1, $G\subseteq \Aut(S)$, with $\rk\Pic(S)^G=1$. Then one of the following holds: 
    \begin{enumerate}
        \item $G$ contains $\gamma, \delta$ or $\gamma\delta$,
        \item $G\simeq C_5$ is generated by 
        $$
        (x_0,x_1,x_2,x_3)\mapsto (x_0,\zeta_5x_1,x_2,x_3),
        $$
        and 
        $$
        S=\{x_3^2 + x_2^3 + ax_0^4x_2 + x_0(bx_0^5 + x_1^5)=0\},
        $$
        for some $a,b\in k$,
        \item $G\simeq C_6$ is generated by 
        $$
        (x_0,x_1,x_2,x_3)\mapsto (x_0,\zeta_6 x_1,x_2,x_3),
        $$
        and
        $$
        S=\{x_3^2 + x_2^3 + ax_0^4x_2 + x_0^6 + bx_1^6=0\},
        $$
        for some $a,b\in k$.
       \end{enumerate}
\end{lemm}

Additionally, we list all cases when $\rk\Pic^G(S)=1$, with generators, for $S$ with maximal automorphism groups. Such $S$ are of types: 
\begin{itemize}
    \item [(I)] $x_3^2+x_2^3+x_0x_1(x_0^4-x_1^4)=0$;
    \item [(II)] $x_3^2+x_2^3+x_0^6+x_1^6=0$;
    \item [(IV)]$x_3^2+x_2^3+x_0(x_0^5+x_1^5)=0$;
    \item [(VII)]$x_3^2+x_2^3+x_2x_0^4+x_1^6=0$;
    \item [(XV)]$x_3^2+x_2^3+x_2(ax_0^4+x_1^4)+x_0^2(bx_0^4+cx_1^4)=0$, where $a,b,c\in k$ are general. 
\end{itemize}

Put
\begin{align*}
    \sigma_1&:(x_0,x_1,x_2,x_3)\mapsto(x_0,x_1,\zeta_3x_2,-x_3),\\
    \sigma_2&:(x_0,x_1,x_2,x_3)\mapsto(\zeta_8^7(x_0+x_1),\zeta_8^5x_0+\zeta_8x_1,2\zeta_3x_2,2\sqrt{2}x_3),\\
    \sigma_3&:(x_0,x_1,x_2,x_3)\mapsto(-\zeta_8x_1,\zeta_8^7x_0,-x_2,\zeta_4x_3),\\
       \sigma_4&:(x_0,x_1,x_2,x_3)\mapsto(\zeta_8x_0,\zeta_8^7x_1,-x_2,\zeta_4x_3),\\
       \sigma_5&:(x_0,x_1,x_2,x_3)\mapsto(\zeta_{12}x_0,\zeta_{12}^{11}x_1,-x_2,\zeta_4x_3),\\
         \sigma_6&:(x_0,x_1,x_2,x_3)\mapsto(x_1,x_0,x_2,x_3),\\
          \sigma_7&:(x_0,x_1,x_2,x_3)\mapsto(\zeta_{10}x_0,\zeta_{10}^{9}x_1,\zeta_5x_2,\zeta_{10}^3x_3),\\
           \sigma_8&:(x_0,x_1,x_2,x_3)\mapsto(\zeta_{24}x_0,\zeta_{24}^{23}x_1,\zeta_{24}^{14}x_2,\zeta_{8}^{7}x_3),\\
           \sigma_9&:(x_0,x_1,x_2,x_3)\mapsto(\zeta_{8}x_0,\zeta_{8}^{7}x_1,\zeta_{4}x_2,\zeta_{8}^{3}x_3).
\end{align*}

The table below lists all cases when $\rk\Pic(S)^G=1$.
\begin{longtable}{|c|c|c|c|}
\hline  
Type &$G$ &GapID& Generators\\\hline
I&$\fC_2$&{\tt(2,1)}&{$\sigma_1^3$}\\\hline
I&$\fC_3$&{\tt(3,1)}&{$\sigma_1^4$}\\\hline
I&$\fC_6$&{\tt(6,2)}&{$\sigma_1^5$}\\\hline
I&$\fC_4$&{\tt(4,1)}&{$\sigma_4^2$}\\\hline
I&$\fC_{12}$&{\tt(12,2)}&{$\sigma_4^2\sigma_1^5$}\\\hline
I&$\fC_2^2$&{\tt(4,2)}&{$\sigma_3, \sigma_1^3$}\\\hline
I&$\fC_6$&{\tt(6,2)}&{$\sigma_3\sigma_1^5$}\\\hline
I&$\fC_2\times \fC_6$&{\tt(12,5)}&{$\sigma_1^3, \sigma_3\sigma_1^4$}\\\hline
I&$\fC_6$&{\tt(6,2)}&{$\sigma_1\sigma_2^2$}\\\hline
I&$\fC_6$&{\tt(6,2)}&{$\sigma_2^5$}\\\hline
I&$\fC_3^2$&{\tt(9,2)}&{$\sigma_2^2, \sigma_1\sigma_2^5$}\\\hline
I&$\fC_3\times \fC_6$&{\tt(18,5)}&{$\sigma_2^4, \sigma_1^2\sigma_2$}\\\hline
I&$Q8$&{\tt(8,4)}&{$\sigma_4\sigma_3, \sigma_4^2$}\\\hline
I&$\fC_3\times Q8$&{\tt(24,11)}&{$\sigma_1\sigma_3\sigma_4, \sigma_3\sigma_4^7$}\\\hline
I&$\fD_4$&{\tt(8,3)}&{$\sigma_4^2, \sigma_3$}\\\hline
I&$\fC_3\times \fD_4$&{\tt(24,10)}&{$\sigma_2\sigma_4\sigma_2, \sigma_3\sigma_1^5$}\\\hline
I&$\fC_8$&{\tt(8,1)}&{$\sigma_3\sigma_4\sigma_3$}\\\hline
I&$\fC_{24}$&{\tt(24,2)}&{$\sigma_1^4\sigma_4^7$}\\\hline
I&$\fD_6$&{\tt(12,4)}&{$\sigma_1\sigma_2^2, \sigma_1^2\sigma_2\sigma_3$}\\\hline
I&$\fC_3\times \fS_3$&{\tt(18,3)}&{$\sigma_3, \sigma_2^4$}\\\hline
I&$\fC_3\times \fS_3$&{\tt(18,3)}&{$\sigma_4\sigma_3\sigma_4, \sigma_2^4$}\\\hline
I&$\fC_6\times \fS_3$&{\tt(36,12)}&{$\sigma_2^5, \sigma_3\sigma_1^4$}\\\hline
I&$SD16$&{\tt(16,8)}&{$\sigma_4\sigma_3, \sigma_3$}\\\hline
I&$\fC_3\times SD16$&{\tt(48,26)}&{$\sigma_1^2\sigma_4^7, \sigma_1^2\sigma_3$}\\\hline
I&$\SL_2(\bF_3)$&{\tt(24,3)}&{$\sigma_4^2, \sigma_1\sigma_2^5$}\\\hline
I&$\SL_2(\bF_3)$&{\tt(24,3)}&{$\sigma_1\sigma_2, \sigma_4^2$}\\\hline
I&$\fC_3\times \SL_2(\bF_3)$&{\tt(72,25)}&{$\sigma_3\sigma_2^5\sigma_4^7, \sigma_2^4$}\\\hline
I&$\GL_2(\bF_3)$&{\tt(48,29)}&{$\sigma_4\sigma_3\sigma_4, \sigma_1\sigma_4^7\sigma_2^5\sigma_3$}\\\hline
I&$\fC_3\times \GL_2(\bF_3)$&{\tt(144,122)}&{$\sigma_2, \sigma_4$}
\\\hline\hline
II&$\fC_2$&{\tt(2,1)}&{$ \sigma_1^3 $}\\\hline
II&$\fC_3$&{\tt(3,1)}&{$ \sigma_1^4 $}\\\hline
II&$\fC_6$&{\tt(6,2)}&{$ \sigma_1^5 $}\\\hline
II&$\fC_2^2$&{\tt(4,2)}&{$ \sigma_1^3, \sigma_5^2\sigma_6\sigma_5^5\sigma_6 $}\\\hline
II&$\fC_6$&{\tt(6,2)}&{$ \sigma_5^3\sigma_1^4 $}\\\hline
II&$\fC_2\times \fC_6$&{\tt(12,5)}&{$ \sigma_1^3, \sigma_5^3\sigma_1^5 $}\\\hline
II&$\fC_2^2$&{\tt(4,2)}&{$ \sigma_1^3, \sigma_6 $}\\\hline
II&$\fC_6$&{\tt(6,2)}&{$ \sigma_1\sigma_6 $}\\\hline
II&$\fC_2\times \fC_6$&{\tt(12,5)}&{$ \sigma_1^3, \sigma_1^2\sigma_6 $}\\\hline
II&$\fC_4$&{\tt(4,1)}&{$ \sigma_5^2\sigma_6\sigma_5^5 $}\\\hline
II&$\fC_{12}$&{\tt(12,2)}&{$ \sigma_5^2\sigma_6\sigma_5^5\sigma_1^5 $}\\\hline
II&$\fC_6$&{\tt(6,2)}&{$ \sigma_5\sigma_6\sigma_5^5\sigma_6 $}\\\hline
II&$\fC_6$&{\tt(6,2)}&{$ \sigma_5^4\sigma_1^5 $}\\\hline
II&$\fC_3^2$&{\tt(9,2)}&{$ \sigma_5^2, \sigma_5^2\sigma_1^2 $}\\\hline
II&$\fC_3\times \fC_6$&{\tt(18,5)}&{$ \sigma_5^4\sigma_1^4, \sigma_1 $}\\\hline
II&$\fD_4$&{\tt(8,3)}&{$ \sigma_5^2\sigma_6\sigma_5^5\sigma_6, \sigma_6 $}\\\hline
II&$\fC_3\times \fD_4$&{\tt(24,10)}&{$ \sigma_6, \sigma_5^3\sigma_1 $}\\\hline
II&$\fC_3\rtimes\fC_4$&{\tt(12,1)}&{$ \sigma_5^4, \sigma_5^2\sigma_6\sigma_5^5 $}\\\hline
II&$\fC_3\times \fC_3\rtimes\fC_4$&{\tt(36,6)}&{$ \sigma_5^2\sigma_1^2, \sigma_5\sigma_1\sigma_6 $}\\\hline
II&$\fC_6$&{\tt(6,2)}&{$ \sigma_5^5\sigma_1^5 $}\\\hline
II&$\fC_2\times \fC_6$&{\tt(12,5)}&{$ \sigma_1^3, \sigma_6\sigma_5^5\sigma_6 $}\\\hline
II&$\fC_2\times \fC_6$&{\tt(12,5)}&{$ \sigma_1^3, \sigma_1^2\sigma_5^5 $}\\\hline
II&$\fC_3\times \fC_6$&{\tt(18,5)}&{$ \sigma_5^2\sigma_1^2, \sigma_5^5 $}\\\hline
II&$\fC_6^2$&{\tt(36,14)}&{$ \sigma_5\sigma_1, \sigma_6\sigma_5\sigma_6\sigma_5^5 $}\\\hline
II&$\fD_6$&{\tt(12,4)}&{$ \sigma_5^2\sigma_6, \sigma_5\sigma_6\sigma_5 $}\\\hline
II&$\fC_3\times \fS_3$&{\tt(18,3)}&{$ \sigma_5^4\sigma_1^4, \sigma_6 $}\\\hline
II&$\fC_6\times \fS_3$&{\tt(36,12)}&{$ \sigma_1^2\sigma_6, \sigma_5^5\sigma_6\sigma_5 $}\\\hline
II&$\fC_3\rtimes\fD_4$&{\tt(24,8)}&{$ \sigma_5^3\sigma_6, \sigma_5^2\sigma_6 $}\\\hline
II&$\fC_3\times \fC_3\rtimes\fD_4$&{\tt(72,30)}&{$ \sigma_6\sigma_5^2, \sigma_5^5\sigma_1^4 $}
\\\hline\hline
IV&$C_2$&{\tt(2,1)}&{$\sigma_1^3 $}\\\hline
IV&$C_3$&{\tt(3,1)}&{$\sigma_1^2 $}\\\hline
IV&$C_6$&{\tt(6,2)}&{$\sigma_1^5 $}\\\hline
IV&$C_5$&{\tt(5,1)}&{$\sigma_7^2 $}\\\hline
IV&$C_{10}$&{\tt(10,2)}&{$\sigma_1^3\sigma_7 $}\\\hline
IV&$C_{15}$&{\tt(15,1)}&{$\sigma_7\sigma_1^4 $}\\\hline
IV&$C_{30}$&{\tt(30,4)}&{$\sigma_7\sigma_1^5 $}
\\\hline\hline
VII&$\fC_2$&{\tt(2,1)}&$ \sigma_1^3 $\\\hline
VII&$\fC_2^2$&{\tt(4,2)}&$ \sigma_8^6, \sigma_1^3 $\\\hline
VII&$\fC_6$&{\tt(6,2)}&$ \sigma_8^4\sigma_1^3 $\\\hline
VII&$\fC_2\times \fC_4$&{\tt(8,2)}&$ \sigma_1^3, \sigma_8^3\sigma_1^3 $\\\hline
VII&$\fC_6$&{\tt(6,2)}&$ \sigma_1^3\sigma_8^{10} $\\\hline
VII&$\fC_2\times \fC_6$&{\tt(12,5)}&$ \sigma_1^3, \sigma_8^2 $\\\hline
VII&$\fC_2\times \fC_{12}$&{\tt(24,9)}&$ \sigma_1^3, \sigma_1^3\sigma_8^{11} $
\\\hline\hline
XV&$\fC_2$&{\tt(2,1)}&$ \sigma_1^3 $\\\hline
XV&$\fC_2^2$&{\tt(4,2)}&$ \sigma_1^3, \sigma_9^2 $\\\hline
XV&$\fC_2\times \fC_4$&{\tt(8,2)}&$ \sigma_1^3, \sigma_1^3\sigma_9^3 $\\\hline

\end{longtable}

\bibliographystyle{alpha}
\bibliography{nonlinear}

\end{document}